\documentclass[11pt,a4paper]{article}
\usepackage[utf8]{inputenc}
\usepackage{amsmath,amssymb,amsfonts,stmaryrd}
\usepackage{enumerate}
\usepackage[margin=3cm]{geometry}
\usepackage{hyperref}
\usepackage{siunitx}

\usepackage{graphicx}
\usepackage{mathtools}
\usepackage{array}
\usepackage{subcaption}
\usepackage{todonotes}
\usepackage{algorithm}
\usepackage{algorithmicx}
\usepackage{algpseudocode}

\usepackage{algorithm}
\usepackage{graphicx}
\usepackage{mathtools}
\usepackage{mathrsfs}
\usepackage{fix-cm}
\usepackage{multirow}
\usepackage{amsmath}
\usepackage{amssymb}
\usepackage{latexsym}
\usepackage{dsfont}
\usepackage{xcolor}
\usepackage{cite}
\usepackage{dsfont}
\usepackage{enumitem}
\usepackage{todonotes}
\usepackage{hyperref}
\usepackage{booktabs}
\usepackage{bbm}
\usepackage{cancel}
\usepackage[normalem]{ulem}

\usepackage{tabularray}
\UseTblrLibrary{booktabs,siunitx}

\newtheorem{experiment}{Experiment}
\let\oldexperiment\experiment
\renewcommand{\experiment}{\oldexperiment\normalfont}

\hypersetup{
colorlinks=true,
linkcolor=blue,
citecolor=darkgreen,
urlcolor=blue}
\definecolor{darkgreen}{RGB}{0,150,0}
\definecolor{C0}{RGB}{31,119,180}
\definecolor{C1}{RGB}{255, 127, 14}
\definecolor{MineShaft}{rgb}{0.188,0.188,0.188}

\def\tto{\;{\lower 1pt \hbox{$\rightarrow$}}\kern -10pt
\hbox{\raise 2pt \hbox{$\rightarrow$}}\;}

\def\hat{\widehat}

\def\tilde{\widetilde}

\def\N{\mathbb{N}}

\newcounter{count}

\DeclareMathOperator{\blkdiag}{{blkdiag}}

\usepackage{framed}
\usepackage{mdframed}
\usepackage{pgfplots}
\pgfplotsset{compat = newest}
\usepackage{caption}
\usepackage{subcaption}

\newcommand{\mM}{m}
\newcommand{\Csta}{a}
\newcommand{\Cstb}{b}
\newcommand{\Cstc}{c}

\newcommand{\salg}{SNSM}

\newcommand{\MoreauYosida}[2]{{\mathtt{e}}_{#2} #1}

\newcommand{\Asp}[2]{{{\mathtt{A}}_{#1 }#2}}

\usepackage{amsthm}
\usepackage{thmtools}
\newtheorem{theorem}{Theorem}[section]

\newtheorem{proposition}[theorem]{Proposition}

\newtheorem{definition}[theorem]{Definition}
\newtheorem{assumption}[theorem]{Assumption}

\theoremstyle{remark}
\newtheorem{remark}[theorem]{Remark}
\newtheorem{example}[theorem]{Example}

\DeclareMathOperator*{\argmax}{argmax}
\DeclareMathOperator*{\argmin}{argmin}

\newcommand{\Rex}{\overline{\mathbb{R}}}
\let\epsilon\varepsilon
\DeclareMathAlphabet{\mathpzc}{OT1}{pzc}{m}{it}

\def\R{\mathbb{R}}
\def\N{\mathbb{N}}

\newcommand{\BB}{\bfseries}

\title{Nonmonotone subgradient methods based on \\ a local descent lemma}

\author{Francisco J.\ Arag\'on-Artacho\thanks{Department of Mathematics,
                             University of Alicante,
                             Alicante, \textsc{Spain}.
                                 Email:~\href{mailto:francisco.aragon@ua.es}
                                 {francisco.aragon@ua.es}}
        \and
       Rub\'en Campoy\thanks{Department of Mathematics,
                             University of Alicante,
                             Alicante, \textsc{Spain}.
                                 Email:~\href{mailto:ruben.campoy.es}
                                 {ruben.campoy@ua.es}}
        \and 
        Pedro P\'erez-Aros\thanks{Departamento de Ingenier\'ia Matem\'atica and Centro de Modelamiento Matem\'atico (CNRS IRL2807), Universidad de Chile, Beauchef 851, Santiago, \textsc{Chile}. 
        Email:~\href{mailto:pperez@dim.uchile.cl}{pperez@dim.uchile.cl}}
         \and
        David Torregrosa-Bel\'en\thanks{Department of Mathematics,
                             University of Alicante,
                             Alicante, \textsc{Spain}.
                                 Email:~\href{mailto:david.torregrosa@ua.es}
                                 {david.torregrosa@ua.es}}
 }

\makeatletter
\renewcommand{\maketitle}{
  \begin{center}
    {\LARGE\bfseries\@title\par}
    \vskip 2em
    {\large
      \lineskip .5em
      \begin{tabular}[t]{c}
        \@author
      \end{tabular}\par}
  \end{center}
  \@thanks
}
\makeatother

\begin{document}

\maketitle

\begin{abstract}
In this paper we present a nonmonotone line search subgradient algorithm tailored to upper-$\mathcal{C}^2$ functions. This is a family of nonsmooth and nonconvex functions that satisfies a nonsmooth and local version of the descent lemma, making them suitable for line searches. We prove subsequential convergence of the proposed algorithm to a stationary point of the optimization problem. Our approach allows us to cover the setting of various subgradient algorithms, including Newton and quasi-Newton methods. In addition, we propose a specification of the general scheme, named Self-adaptive Nonmonotone Subgradient Method (SNSM), which automatically updates the parameters of the line search. Particular attention is paid to the minimum sum-of-squares clustering problem, for which we provide a concrete implementation of SNSM. We conclude with some numerical experiments where we exhibit the advantages of SNSM in comparison with some known algorithms.
\end{abstract}


\section{Introduction}\label{intro_old}

In this work, we propose a nonmonotone subgradient algorithm for solving the general nonconvex unconstrained program
\begin{equation}\label{eq:P}\tag{P}
	\min_{x \in \mathbb{R}^n} \varphi(x),
\end{equation}
where the objective function \( \varphi: \mathbb{R}^n \to \mathbb{R} \) is assumed to be upper-\( \mathcal{C}^2 \). Roughly speaking, these functions admit a local representation as a minimum of twice-continuously differentiable functions (see Definition~\ref{def:upperC2}). Upper-$\mathcal{C}^2$ functions  satisfy a nonsmooth local version of the descent lemma, a property that makes them suitable for extending the framework of line search methods (for details, see Section~\ref{sect:upperC2}). This is important in practice, since upper-$\mathcal{C}^2$ functions naturally arise in different applications, as data mining (see Section~\ref{sect:appl}).

Upper-$\mathcal{C}^2$ functions can equivalently be characterized as those that locally admit a difference of convex (DC) decomposition consisting of a smooth and a possibly nonsmooth convex function (see Proposition~\ref{p:deslemma}). This might lead one to think that minimizing an upper-$\mathcal{C}^2$ function can be reduced to a standard DC programming task. However, this approach presents practical challenges. On the one hand, not all upper-$\mathcal{C}^2$ functions are globally DC (see Example~\ref{example01}). On the other hand, it requires prior knowledge of a suitable DC decomposition. For instance, given a DC function $f=g_1-g_2$ for some convex functions $g_1$ and $g_2$, it is often unclear which decomposition is optimal for numerical purposes, since $f=(g_1+h)-(g_2+h)$ yields another valid DC decomposition for any convex function $h$. This difficulty can be circumvented by treating the function $f$ as a unified entity, thereby eliminating the need for a specific DC decomposition.

Let us recall that descent line search methods are a class of iterative optimization algorithms designed to find (local) minima of differentiable functions. Their core technique consists in updating the current iterate by moving in a direction that reduces the value of the objective function. Hence, given a starting point $x_0\in\R^n$, a general iteration step is defined by
\[
x_{k+1} = x_k + \tau_k d_k,
\]
where $\tau_k>0$ is a stepsize (also known as learning rate in the machine learning community) and $d_k\in\R^n$ is a descent direction. One of the most fundamental choices is to take $d_k = -\nabla \varphi (x_k)$, with $\nabla \varphi$ denoting the gradient of the objective function, leading to the well-known \textit{gradient descent} method. While simple and broadly applicable, gradient descent can suffer from slow convergence. To improve convergence rates, descent methods can incorporate second-order information. This leads to the development of \textit{Newton} and \textit{quasi-Newton} methods, where the descent direction $d_k$ is computed by solving the linear system
\[
- \nabla \varphi(x_k) = B_k d_k,
\]
where $B_k$ is a symmetric positive definite matrix approximating (or equal to) the Hessian of the function. In Newton's method, $B_k$ is the exact Hessian, while quasi-Newton methods use an approximation built from gradient evaluations across iterations (see, e.g., \cite{MR3890045,MR2244940} and the references therein).

While most DC algorithms only use first-order information, a coderivative-based semi-Newton method for difference programming has been recently proposed in~\cite{aragonartacho2023coderivativebased}. When applied to a DC decomposition whose first component is twice differentiable, this method employs a Newton-type descent direction based on the Hessian of that component. However, since second-order information of the second function is not used, the method fails to capture the full curvature of the objective function. This limitation is evident in the minimum sum-of-squares clustering problem (see Section~\ref{sect:appl}). For this problem, we demonstrate that effective second-order information can be successfully extracted and utilized by working directly with the upper-$\mathcal{C}^2$ function, which yields significantly faster algorithmic schemes (see Subsection~\ref{subsec:SOS}). This further motivates our development of numerical schemes tailored to upper-$\mathcal{C}^2$ functions, avoiding the need to rely on DC decompositions.

In contrast to classical methods, nonmonotone descent schemes allow occasional increases in the objective function during the optimization process. These methods relax the strict requirement of monotonic decrease, provided that overall progress toward a minimum is ensured. The procedure was first introduced in the context of Newton-type methods in \cite{MR849278}, and has since been extended and successfully applied in broader settings. Nonmonotone strategies are particularly effective for nonconvex or large-scale problems, where traditional methods may stagnate. For further developments and applications, see, e.g.,~\cite{MR1883074,MR2112963,MR993007,MR1430555,MR2903511} and the references therein.

In this manuscript, we propose a nonmonotone scheme for problem~\eqref{eq:P}. Our approach is inspired by the seminal work~\cite{MR849278}, which is adapted to the current setting where the objective function may not be differentiable. To this end, we replace standard gradient information with subgradients. Notably, our method allows for any choice of subgradient. Indeed, for the class of \mbox{upper-$\mathcal{C}^2$} functions, we show that any subgradient can be used to construct a local quadratic upper approximation of the objective function. This leads to a generalized version of the classical descent lemma suited to our nonsmooth setting.

Although nonmonotone line searches bring opportunities to improve the performance of optimization algorithms, they also require the tuning of additional parameters, such as the line search stepsize (which is initially chosen at each iteration) and the memory parameter (which controls how many past function values are taken into account). A wrong adjustment of these parameters may slow down the resulting algorithm, and one may lose all possible benefits of nonmonotone line searches. With this potential weakness in mind, we propose a specification of the main scheme, that we name the \emph{Self-adaptive Nonmonotone Subgradient Method} (\emph{\salg}), which incorporates an automatic procedure for selecting the stepsize and memory parameters of the line search. Other possibilities that would be interesting to investigate include the nonmonotone strategy proposed by Zhang and Hager~\cite{MR2112963},  which relies on a weighted average of all past function values in contrast to the maximum over a number of past iterations used here.

The main contributions of this work are summarized as follows:

\begin{itemize}
	\item We characterize the variational properties that make upper-$\mathcal{C}^2$ functions suitable for optimization, allowing us to subsume  existing approaches within a unified framework. Furthermore, we demonstrate their prevalence in optimization and provide examples to facilitate their identification.
	
	\item We present a line search subgradient method tailored to upper-$\mathcal{C}^2$ functions (Algorithm~\ref{alg:1}), which permits the incorporation of second-order information. We prove subsequential convergence to stationary points of problem~\eqref{eq:P} from arbitrary starting points, under the stated assumptions. Additionally, we propose a self-adaptive version of this general method (Algorithm~\ref{alg:SNSM}) to mitigate the effects of suboptimal parameter configurations.
	
	\item We conduct extensive numerical comparisons against specialized algorithms  for the minimum sum-of-squares clustering problem. The experiments show that incorporating second-order information---which is typically lost in standard DC approaches---results in significantly superior algorithmic performance. Overall, our experiments indicate that our method is competitive with modern large-scale solvers.
\end{itemize}

The remainder of the paper is organized as follows. In Section~\ref{sect:preliminary}, we present some preliminary concepts and fix the notation. The class of upper-$\mathcal{C}^2$ functions is thoroughly introduced in Section~\ref{sect:upperC2}. Several examples are given, showing the breadth of this family of functions. In Section~\ref{sect:nlsm}, we present our nonmonotone line search method tailored to upper-$\mathcal{C}^2$ functions, and we introduce its specification SNSM in Section~\ref{sect:sa}.
The applicability of the proposed scheme to a problem in data mining, namely, the minimum sum-of-squares clustering problem, is detailed in Section~\ref{sect:appl}, where we specify a concrete implementation of \salg\ for this setting. Next, we report in Section~\ref{sect:numerics} some numerical experiments on this problem, illustrating the competitive performance of our method. Finally, some conclusions are drawn in Section~\ref{sect:conc}.

\section{Preliminaries and notation}\label{sect:preliminary}
Throughout this paper, we write $\R^n$  for the Euclidean space of dimension~$n$, and the Euclidean norm and inner product are denoted by~$\| \cdot\|$ and~$\langle \cdot, \cdot \rangle$, respectively.
%The set of $n\times n$ symmetric matrices is denoted by $\mathbb{S}^n$, while $\mathbb{S}_{++}^n$ stands for the cone of positive definite matrices.
We use $\lambda_{\max}(B)$ for the largest eigenvalue of a symmetric matrix~$B$, and $\lambda_{\min}(B)$ for the smallest. Given $a\in\R$, we write $[a]^+$  to denote its positive part.

Given some constant $L\geq 0$, a vector-valued function $F: C\subseteq \mathbb{R}^n \to \R^m$ is said to be  $L$-\emph{Lipschitz continuous} on $C$ if
\begin{equation*}
	\|F(x)-F(y)\| \leq L \|x-y\|,\quad \text{for all } x,y\in C,
\end{equation*}
and \emph{locally Lipschitz continuous} around $\bar x\in C$ if it is Lipschitz continuous in some neighborhood of $\bar x$.

A function $f:\R^n\to\R$ is said to be of class $\mathcal{C}^k$ if it is differentiable and all its partial derivatives up to order $k$ are continuous. In particular, we say that a function $f$ is  \emph{smooth} if it is of class $\mathcal{C}^1$, i.e., it  is differentiable with continuous gradient. If, in addition, its gradient is $L$-Lipschitz continuous we say that $f$ is $L$-smooth. The class of differentiable functions with a locally Lipschitz continuous gradient is denoted by $\mathcal{C}^{1,+}$.  We represent the gradient and Hessian  of $f$ as $\nabla f$ and  $\nabla^2 f$, respectively.

Given a locally Lipschitz continuous function $f$ around $\bar x\in \mathbb{R}^n$, the \emph{Clarke subdifferential} of $f$ at $\bar{x}$ is defined as
\[
\partial f(\bar{x}):= \left\{   v \in \mathbb{R}^n :  \langle v, h \rangle \leq f^\circ (\bar{x};h) \text{ for all } h \in \R^n			 \right\},
\]
where $f^\circ (\bar{x};h) $ stands for the  \emph{Clarke directional derivative} at $\bar x$ with respect to the direction $h$, that is,
\begin{align*}
	f^\circ (\bar{x};h) := \limsup_{x \to \bar x, t \to  0^+ } \frac{  f(x + th ) - f(x)    }{t}.
\end{align*}

Let $f:\R^n \to\R $ be a locally Lipschitz function around $\bar x$.  	We say that  $f$ is \emph{prox-regular} at a  point $\bar{x}$ if for every $\bar{v}\in\partial f(\bar{x})$ there exist $\varepsilon>0$ and $\rho>0$ such that
\[
f(\tilde{x}) \geq f(x) +  \langle v, \tilde{x}-x  \rangle - \frac{\rho}{2}\|\tilde{x}-x\|^2, \text{ for all } \|\tilde{x}- \bar{x}\|\leq \varepsilon,
\]
when $v\in\partial f(x)$, $\|v-\bar{v}\|<\varepsilon$, $\|x-\bar{x}\|<\varepsilon$ and $f(x) < f(\bar{x}) +\varepsilon$. When   $f$ is prox-regular at every point, we simply say that $f$ is prox-regular (see, e.g., \cite{MR3823783,MR1491362,zbMATH00914945,zbMATH07502906} for more details).

Finally,  the \emph{upper Dini directional derivative} of $f:\R^n\to\R$ at some point $\bar{x}\in\R^n$ in the direction $d\in\R^n$ is defined as
$$d^{+} f(\bar{x} ; d):=\limsup_{t \to 0^+} \frac{f(\bar{x}+t d)-f(\bar{x})}{t}.$$

\section{Upper-$\mathcal{C}^2$ functions and the descent lemma} \label{sect:upperC2}

Most known gradient-based algorithms are designed to ensure descent of the objective function per iteration. Many of the algorithms proposed in the literature that  rely exclusively on first- and second-order information are developed for the setting of differentiable functions with at least locally Lipschitz continuous gradients. This class of functions is particularly important in both theoretical and practical optimization. Its significance lies in the fact that such functions satisfy the so-called \emph{descent lemma} (see, e.g.,~\cite[Lemma~A.11]{MR3289054}). More precisely, for a differentiable function
$
f:C \subseteq \mathbb{R}^n \to \mathbb{R}
$
with an \mbox{$L$-Lipschitz} continuous gradient on a nonempty open convex set $C$, it holds that
\begin{equation}\label{eq:descentlemmaineq}
	f(y) \leq f(x) + \langle \nabla f(x), y - x \rangle + \frac{L}{2} \|y - x\|^2, \quad \text{for all } x, y \in C.
\end{equation}
Recall that  $L$-smooth functions are those with global Lipschitz continuous gradients. Consequently, these functions  verify the descent lemma~\eqref{eq:descentlemmaineq} for \mbox{$C=\mathbb{R}^n$}.
In contrast, for smooth functions whose gradients are only locally Lipschitz continuous on $\mathbb{R}^n$ (i.e., functions belonging to the class $\mathcal{C}^{1,+}$), the descent lemma may hold only locally around every point of the space. In such a case, for each $\bar{x}\in\mathbb{R}^n$, the descent lemma yields
\begin{equation*}
	f(y) \leq f(x) + \langle \nabla f(x), y - x \rangle + \frac{L(\bar{x})}{2} \|y - x\|^2, \quad \text{for all } x, y \in U(\bar{x}),
\end{equation*}
where $U(\bar{x})$ is an open convex neighborhood  of $\bar{x}$ and $L(\bar{x})\geq 0$ is a constant depending on both $\bar{x}$ and $U(\bar{x})$.

In recent works, certain classes of nonsmooth functions have been successfully incorporated into descent algorithms (see, e.g., \cite{MR4078808,Aragon2020,ferreira2021boosted,aragonartacho2023boosted,aragonartacho2023coderivativebased}). These studies demonstrate that, for such classes of mappings, a sufficient decrease condition analogous to the classical descent lemma can be established by replacing the gradient with a suitable subgradient in~\eqref{eq:descentlemmaineq}.

The work~\cite{aragonartacho2023coderivativebased} investigated objective functions that can be expressed as the difference of two functions, namely $\varphi := g - h$, where $g$ is differentiable with a locally Lipschitz continuous gradient, and $h$ is a locally Lipschitz continuous prox-regular function (see, e.g., \cite[Definition 13.27]{MR1491362}).  The authors developed a descent method for the objective function $\varphi$, where descent directions are generated using coderivative information. Among the results presented, they proved that this class of functions satisfies a local version of the descent lemma (see~\cite[Lemma~3.6]{aragonartacho2023coderivativebased}).

Later, the authors of~\cite{aragonartacho2023boosted} proposed a proximal algorithm for structured problems of the form
\begin{equation*}
	\min_{x\in\mathbb{R}^n} f(x) +g(x)-\sum_{i =1 }^p h_i(\mathrm{\Psi}_i (x)),
\end{equation*}
where the function $f:\R^n\to\R$ is locally Lipschitz, $g:\R^n\to{\Rex:={]-\infty,+\infty]}}$ is lower-semicontinuous and prox-bounded, $h_i:\R^{m_i}\to\R$ are convex and continuous functions, and $\mathrm{\Psi}_i:\R^n\to\R^{m_i}$ are differentiable functions with Lipschitz continuous gradients, for $i\in\{1,\ldots,p\}$. In that setting, they showed that it is possible to consider subgradient information of the function $f$, provided that this function  belongs to the class of $\kappa$-upper-$\mathcal{C}^2$ functions, for some $\kappa\geq 0$. This means that $f$ can be expressed as the optimal value of a family of quadratic functions over a compact space $C$, in the sense that
\[
f(x) = \min_{c \in C} \left\{ \kappa\|x\|^2 - \langle a(c), x \rangle - b(c)\right\},\text{ for all }x \in \mathbb{R}^n,
\]
where $a$ and $b$ are continuous functions. It was shown in~\cite{aragonartacho2023boosted} that such functions satisfy a variant of the inequality~\eqref{eq:descentlemmaineq}, where the role of the gradient is replaced by any subgradient, and the constant $\frac{L}{2}$ in the quadratic term is precisely the parameter $\kappa$.

Let us recall a related definition extracted from the book of Rockafellar and Wets (see~\cite [Definition~10.29]{MR1491362}) that will serve us to encompass both frameworks~\cite{aragonartacho2023coderivativebased} and~\cite{aragonartacho2023boosted}. It refers to a class of functions that is usually known as \emph{lower-$\mathcal{C}^2$} functions (correspondingly, a function $f$ is called \emph{upper-$\mathcal{C}^2$} if $-f$ is lower-$\mathcal{C}^2$). The properties of these functions have been well-studied in the context of variational analysis (see, e.g.,~\cite{MR1363364,MR2642725} and the references therein). However, to the best of our knowledge, this class has not received much attention in the numerical optimization literature. Beyond the already mentioned exceptions, we are only aware of~\cite{MR4635008}, where a sequential quadratic programming trust-region method is proposed for the problem of minimizing  a Lipschitz continuous and upper-$\mathcal{C}^2$ function in a bounded region including $L$-smooth equality constraints, and of~\cite{deOliveira2023}, where the author studies a Frank–Wolfe method for nonsmooth upper-$\mathcal{C}^{1,\alpha}$ functions, relying on a generalized descent lemma for this class (see~\cite[Lemma~5]{deOliveira2023}). It is also worth mentioning that, in~\cite[Chapter~14]{WWBook}, the authors explore the structure of differences of weakly convex functions as upper-bounded models for nonsmooth optimization problems.

%, which is motivated from an application to power flow optimization.
%with the already mentioned exceptions. Even in~\cite{aragonartacho2023coderivativebased}  the term  upper-$\mathcal{C}^2$  itself was not used explicitly.

\begin{definition}[Upper-$\mathcal{C}^2$ functions]\label{def:upperC2}
	Let $U$ be an open set of $\R^n$. We say that a function $\varphi:U \to \R$ is \emph{upper-$\mathcal{C}^2$} on $U$, if on some  neighborhood $V$ of each $\bar{x}\in U$ there is a representation
	\[
	\varphi(x) = \min_{c\in C} \varphi_c(x),
	\]
	where the functions $\varphi_c$ are of class $\mathcal{C}^2$ on $V$, and $C$ is a compact set (in some topological space) such that $\varphi_c$ and its first- and second-order partial derivatives depend  continuously on $(x,c)\in V\times C$. Further, we say that $\varphi$ is \emph{upper-$\mathcal{C}^2$} around $\bar x$, provided that it is upper-$\mathcal{C}^2$ on some neighborhood of $\bar{x}$.
\end{definition}

It is easy to see that a function $\varphi$ is upper-$\mathcal{C}^2$ on an open set $U$ if and only if it is  upper-$\mathcal{C}^2$ around any point $x\in U$. Moreover, by~\cite[Theorem~10.31]{MR1491362}, these functions are locally Lipschitz. The next result summarizes some useful characterizations of upper-$\mathcal{C}^2$ functions.
\begin{proposition}\label{p:deslemma}
	Let $U$ be an open set such that $\varphi:\R^n\to\R$ is locally Lipschitz on $U$. The following assertions are equivalent:
	\begin{enumerate}[label=(\roman*)]
		\item\label{it:p25-0} $\varphi$ is upper-$\mathcal{C}^2$ on $U$;
		\item\label{it:p25-i} for each $\bar{x}\in U$, there exist a constant $\kappa\geq 0$, some neighborhood $V$ of $\bar x$, a compact set $C$ (in some topological space) and some continuous functions $a:C\to\R^n$ and $b:C\to\R$ such that
		\begin{align}\label{eq:reprupperC2}
			\varphi(x) = \min_{c\in C} \left\{ \kappa \|x\|^2 - \langle a(c),x\rangle - b(c) \right\}, \quad \text{for all } x\in V;
		\end{align}

		\item\label{it:p25-iii} for each $\bar{x}\in U$, there exist a constant $\kappa\geq 0$ and some neighborhood $V$ of $\bar x$ such that for all $x\in V$ and  $w \in \partial \varphi(x)$, it holds
		\begin{align}\label{equpperdesc}
			\varphi(y) \leq \varphi(x) + \langle w , y-x\rangle + \kappa \|y -x\|^2,\quad\text{for all } y \in V;
		\end{align}
		\item\label{it:p25-iv}   for each $\bar{x}\in U$, there exists  some neighborhood $V$ of $\bar x$ where $\varphi$ can be expressed as
		$\varphi = g-h$, where $g$ is differentiable with  Lipschitz gradient, and $h$  is  Lipschitz  and  prox-regular. Indeed, one can take $g=\kappa \| \cdot \|^2$, for some $\kappa\geq 0$,
		and $h$ to be convex.
		
	\end{enumerate}
\end{proposition}

\begin{proof}
	The equivalence between \ref{it:p25-0} and \ref{it:p25-i} is deduced from \cite[Theorem~10.33]{MR1491362} applied to $-\varphi$.
	To prove the equivalence \ref{it:p25-i}$\iff$\ref{it:p25-iii}, fix some $\bar{x}\in U$ and an open convex neighborhood $V\subset U$ of $\bar{x}$, and apply \cite[Proposition~2.2]{aragonartacho2023boosted} to the set $V$.  Finally, let us notice that \eqref{eq:reprupperC2} allows us to write
	\begin{align*}
		\varphi(x) = g(x)- h(x), \quad \text{for all } x \in V,
	\end{align*}
	where $g(x):= \kappa \|x\|^2$ and $h(x) := \max_{ c\in C}\{   \langle a(c),x\rangle + b(c)\}$.  This shows the implication \ref{it:p25-i} $\Rightarrow $ \ref{it:p25-iv}. Finally, we conclude the proof by showing that \mbox{\ref{it:p25-iv} $\Rightarrow$ \ref{it:p25-0}}. Let us suppose that there exists some neighborhood $V$ of $\bar x$ and a representation $\varphi = g- h$, where $g$ is  differentiable with  Lipschitz gradient on $V$  and $h$ is Lipschitz continuous and prox-regular on $V$. It follows from \cite[Proposition 13.34]{MR1491362}  that $g$ is upper-$\mathcal{C}^2$ on $V$, and by \cite[Proposition~13.33]{MR1491362},  the function $-h$ is upper-$\mathcal{C}^2$ on $V$. This concludes the proof, since the class of upper-$\mathcal{C}^2$ functions is preserved under addition.
\end{proof}

It is worth noting that Proposition~\ref{p:deslemma} characterizes the class of \mbox{upper-$\mathcal{C}^2$} functions as those satisfying a nonsmooth and local version of the   descent lemma, given by~\eqref{equpperdesc}. This property will  be used to prove the convergence of the algorithmic schemes that we propose in Sections~\ref{sect:nlsm} and~\ref{sect:sa}. In addition, Proposition~\ref{p:deslemma} shows that the nonsmooth functions considered in \cite{aragonartacho2023coderivativebased,aragonartacho2023boosted} are upper-$\mathcal{C}^2$ functions. Moreover, inequality~\eqref{equpperdesc} can be interpreted as a local counterpart of \cite[Proposition~2.2]{aragonartacho2023boosted}. Unlike that global result, Proposition~\ref{p:deslemma} permits the constant $\kappa$ to vary with the point $\bar{x} \in U$. Notably, while the algorithm developed  in \cite{aragonartacho2023boosted} requires prior knowledge of $\kappa$, the subgradient methods that we introduce below  do not rely on such information. Furthermore, Proposition~\ref{p:deslemma} provides a more precise and localized version of \cite[Lemma~3.6]{aragonartacho2023coderivativebased}, where the descent property was established using a localization of the subgradient.

Let us present some examples that reveal that the class of upper-$\mathcal{C}^2$ functions naturally appears in several contexts in  optimization. While there is no universal criterion to determine whether a function is \mbox{upper-$\mathcal{C}^2$}, the first two examples provide a useful toolkit for this purpose. Specifically, the first example illustrates the preservation of the property under addition, multiplication by scalars, and composition.

\begin{example}[Operations over upper-$\mathcal{C}^2$ functions]
	It is easy to prove that the class of upper-$\mathcal{C}^2$ functions is stable under addition and multiplication by positive scalars (see, e.g., \cite[Exercise 10.35]{MR1491362}).  It is also closed under composition of upper-$\mathcal{C}^2$ and $\mathcal{C}^{1,+}$ functions. More precisely,  given a function $\Phi$ that is $\mathcal{C}^{1,+}$ around~$\bar x$, and an  upper-$\mathcal{C}^2$ function $\varphi   $ around $\Phi(\bar x)$, then the function $ \varphi \circ \Phi$ is upper-$\mathcal{C}^2$ around $\bar x$.  Indeed, it follows from Proposition \ref{p:deslemma} that
	$\varphi(u)=   \kappa \|u\|^2 -h(u)$  around some neighborhood $V$ of $\Phi(\bar x)$. Then,  on some neighborhood of~$\bar x$, one has $\varphi ( \Phi(x)) = \kappa \|\Phi(x)\|^2 -h(\Phi(x)) $, which is the difference of a differentiable function having a locally Lipschitz gradient and a prox-regular function, since $h\circ\Phi$ is prox-regular due to \cite[Proposition~2.3]{MR2069350}.
\end{example}
The second example shows that some of the difference of convex functions studied in the literature also belong to the class of upper-$\mathcal{C}^2$ functions.
\begin{example}[Difference of a $\mathcal{C}^{1,+}$ function and a convex function]\label{example01}
	First, let us notice that  Proposition~\ref{p:deslemma}\ref{it:p25-iv} tells us that the difference of a $\mathcal{C}^{1,+}$  function and a continuous convex function (actually a prox-regular function) is indeed  upper-$\mathcal{C}^2$. In the particular case where the differentiable function is also convex, this encompasses a class of  difference of convex functions whose minimization has been addressed by several numerical methods in recent years (see, e.g., \cite{ferreira2021boosted,Artacho2019,AragonArtacho2018} and the references therein). Nonetheless,  such works crucially rely on this decomposition to be available, which is not always easy to disclose (see, e.g.,~\cite{MR873269,MR1377085}). Finally, let us emphasize that not every upper-$\mathcal{C}^2$ function is necessarily DC, since its domain need not be convex. For instance, the function $\varphi : \mathbb{R} \setminus \{0\} \to \mathbb{R}$ defined by
	\[
	\varphi(x)=x^2+\frac{1}{x^2}
	\]
	is $\mathcal{C}^{1,+}$  around every point of its domain, but it cannot be decomposed (globally) as the difference of two convex functions.
\end{example}

In our third example we recall that the  Moreau and forward-backward envelopes are upper-$\mathcal{C}^2$ functions.

\begin{example}[Moreau and forward-backward envelopes]\label{ex:FBE}
	The  {Moreau envelope} is a   technique used to regularize a potentially nonsmooth or nonconvex function by infimally convolving it with the squared Euclidean norm. Formally, for a given proper lower semicontinuous function $\psi: \mathbb{R}^n \to \Rex$ and a parameter $\lambda > 0$, the \emph{Moreau envelope} $\MoreauYosida{\psi}{\lambda}$ is defined by
	\begin{equation}\label{moreau-envelope}
		\MoreauYosida{\psi}{\lambda} (x) = \inf_{z \in \mathbb{R}^n} \left\{ \psi(z) + \frac{1}{2\lambda} \|x - z\|^2 \right\}.
	\end{equation}
	
	This envelope plays a central role in mathematical programming and various branches of applied mathematics. Over the past few decades, its properties have been widely exploited in the development of numerical algorithms. In particular, it is well known that if $\psi$ is convex, then the Moreau envelope $\MoreauYosida{\psi}{\lambda}$ is continuously differentiable and has a (globally) Lipschitz continuous gradient. More generally, when $\psi$ is not necessarily convex, it  has been shown (see, e.g., \cite[Example 10.32]{MR1491362} or \cite[Proposition 2.7]{aragonartacho2023coderivativebased})  that    the Moreau envelope is an upper-$\mathcal{C}^2$ function provided that  $\lambda \in (0, \lambda_\psi)$, where
	\[
	\lambda_\psi := \sup \left\{ \lambda > 0 \;\middle|\; \MoreauYosida{\psi}{\lambda}(x) > -\infty \text{ for some } x \in \mathbb{R}^n \right\}
	\]
	is  the so-called \emph{threshold of prox-boundedness} of $\psi$. In more exact terms, for every $\lambda \in (0, \lambda_\psi)$, the Moreau envelope, given by~\eqref{moreau-envelope}, admits the representation
	\begin{equation} \label{more-asp}
		\MoreauYosida{\psi}{\lambda}(x) = \frac{1}{2\lambda} \|x\|^2 - \Asp{\lambda}{\psi}(x), \quad x \in \mathbb{R}^n,
	\end{equation}
	where  $\Asp{\lambda}{\psi}(x)$  is the \emph{Asplund function}, which is the convex function defined as
	\begin{equation} \label{asp}
		\Asp{\lambda}{\psi}(x) := \sup_{z \in \mathbb{R}^n} \left\{ \frac{1}{\lambda} \langle z, x \rangle - \psi(z) - \frac{1}{2\lambda} \|z\|^2 \right\}.
	\end{equation}
	
	Now, let us consider the structured nonlinear program given by
	\begin{equation}\label{ProFBE}
		\min_{x\in \mathbb{R}^n} \varphi(x) := f(x) + \psi(x),
	\end{equation}
	where $f\colon \mathbb{R}^n \to \mathbb{R}$ is a $\mathcal{C}^{2,1}$ function (twice continuously differentiable with a globally Lipschitz Hessian) with an $L_f$-Lipschitz continuous gradient, and $\psi\colon \mathbb{R}^n \to \Rex$ is a proper, extended-real-valued prox-bounded function with threshold $\lambda_\psi > 0$. Patrinos and Bemporad introduced in~\cite{Patrinos2013}  a new class of ad-hoc regularization techniques for handling problem~\eqref{ProFBE}, which is now commonly called the forward-backward envelope (see, e.g.,  \cite{MR3845278} and the references therein). Precisely, the \emph{forward-backward envelope} of the function $\varphi$ with parameter $\lambda >0$ is defined as
	\begin{align}\label{FBE}
		\varphi_\lambda(x) :=\inf_{z \in \mathbb{R}^n}\Big\{ f(x) + \langle \nabla f(x),z-x\rangle + \psi(z) +\frac{1}{2\lambda }\| z- x\|^2\Big\}.
	\end{align}
	The forward-backward envelope is clearly a generalization of the Moreau envelope, since $\varphi_\lambda=\MoreauYosida{\psi}{\lambda}$ when $f=0$.
	Further, it was shown in \cite[Theorem~5.5]{aragonartacho2023coderivativebased} that $\varphi_\lambda=g-h$,
	where $g(x):= f(x) +\frac{1}{2\lambda} \|x\|^2 $ is of class $\mathcal{C}^{2,1}$, and
	$h(x):= \langle \nabla f(x),x\rangle + \Asp{\lambda}{\psi}(x- \lambda \nabla f(x))$ is locally Lipschitzian and prox-regular on $\R^n$ for $\lambda \in (0,\lambda_{\psi})$.
	This proves that the forward-backward envelope $\varphi_\lambda$ is also upper-$\mathcal{C}^2$, provided that $\lambda \in (0,\lambda_{\psi})$.
	%		Furthermore, in \cite[Theorem 5.5]{aragonartacho2023coderivativebased} the authors  provided calculus rules for $\MoreauYosida{\psi}{\lambda}$, and relations between stationary points of \eqref{ProFBE} and stationary points of $\MoreauYosida{\psi}{\lambda}$.
	%
	%		
\end{example}

Finally, let us notice  also that a general augmented Lagrangian function belongs to the class of upper-$\mathcal{C}^2$ functions.

\begin{example}[Augmented Lagrangian method]
	The Augmented Lagrangian Method (ALM) is a powerful framework for solving composite and constrained optimization problems, particularly those involving nonlinear or nonconvex constraints. Originally introduced independently by Hestenes \cite{Hestenes1969} and Powell \cite{Powell1969} in the context of nonlinear programming with equality constraints, ALM was later extended by Rockafellar \cite{Rockafellar1973} to  convex programming problems with inequality constraints. This extension significantly broadened the method’s applicability and laid the groundwork for modern augmented Lagrangian techniques. Over the past decades, ALM has become a widely used method in nonlinear optimization, with various adaptations developed to accommodate specific problem structures.
	
	To illustrate the method, let us consider the general optimization problem
	
	\begin{equation}\label{eq:P_ALM}
		\min_{x\in \mathbb{R}^n} \varphi(x):= f (x) + \psi(G(x)),
	\end{equation}
	where $f \colon   \mathbb{R}^n \to \R$ and  $G\colon \mathbb{R}^n \to \R^m$ are $\mathcal{C}^2$, and $\psi\colon \mathbb{R}^m \to \Rex$ is a proper and prox-bounded function with  threshold $\lambda_\psi$. Following \cite{HangSarabi2025,MR4331979},   the augmented Lagrangian function $\mathcal{L} \colon \mathbb{R}^n \times \mathbb{R}^m \times (0,+\infty) \to \Rex$  of problem~\eqref{eq:P_ALM} is given by
	\begin{align*}
		\mathcal{L} (x,y, \lambda) :=\inf_{ u \in \mathbb{R}^m}\left\{ f(x) +  	\psi( G(x) + u) + \frac{1}{2\lambda }\| u\|^2 - \langle y, u\rangle 	\right\}.
	\end{align*}
	By iteratively updating the primal and dual  variables, the ALM balances the enforcement of feasibility and optimality.  Namely, given a point $(x_k,y_k, \lambda_k)$, it generates variables $(x_{k+1},y_{k+1})$ by the rule
	\begin{align*}
		x_{k+1} &\in \argmin_{ x\in \mathbb{R}^n} 	\mathcal{L} (x,y_k, \lambda_k),\\
		y_{k+1} &\in y_k + \lambda_k \partial_{y} 	\mathcal{L} (x_{k+1},y_k, \lambda_k).
	\end{align*}
	In order to show that the  augmented Lagrangian function is upper-$\mathcal{C}^2$, let us derive an alternative representation in terms of the Moreau envelope of $\psi$. By completing the squares on the terms $ \frac{1}{2\lambda }\| u\|^2 - \langle y, u\rangle $ and performing the change of variables $v:=G(x)+u$, we get
	\[
	\begin{aligned}
		\mathcal{L} (x,y, \lambda) & =\inf_{ u \in \mathbb{R}^m}\left\{ f(x) +  	\psi( G(x) + u) + \frac{1}{2\lambda }\| u - \lambda y \|^2 -  \frac{\lambda}{2} \|y\|^2 	\right\} \\
		& = f(x) + \inf_{u \in \mathbb{R}^m}\left\{ \psi (G(x)+u) + \frac{1}{2\lambda }\| u - \lambda y \|^2 \right\} - \frac{\lambda}{2} \|y\|^2 \\
		& = f(x) + \inf_{v \in \mathbb{R}^m} \left\{ \psi (v) + \frac{1}{2\lambda }\| v  - (G(x) + \lambda y) \|^2 \right\} -  \frac{\lambda}{2} \|y\|^2 \\
		& = f(x) +  \MoreauYosida{\psi}{\lambda} (G(x) + \lambda y) - \frac{\lambda}{2} \|y\|^2.
	\end{aligned}
	\]
	Therefore, by the preceding examples, the  augmented Lagrangian function is indeed upper-$\mathcal{C}^2$ for $\lambda \in (0,\lambda_{\psi})$.
\end{example}

\section{A nonmonotone subgradient method for nonconvex minimization}\label{sect:nlsm}

Let us first briefly draw our attention to the  optimality conditions for~\eqref{eq:P}. If $\bar{x}\in\R^n$ is a local minimum, it follows by Fermat's rule  (see, e.g., \cite[Proposition~1.30]{MR3823783}) that
\begin{equation}\label{eq:st-Clarke}
	0\in\partial\varphi(\bar{x}).
\end{equation}
We say that a point $\bar{x}$ satisfying~\eqref{eq:st-Clarke} is a \emph{stationary point} for problem~\eqref{eq:P}.

We describe the pseudocode of our nonmonotone line search procedure for the minimization of upper-$\mathcal{C}^2$ functions in Algorithm~\ref{alg:1}. It follows the structure of classical line search methods for minimizing smooth functions, where the gradient is now replaced by a subgradient $w_k\in\partial \varphi (x_k)$. Once an appropriate search direction $d_k\in\R^n$ is selected, it performs an Armijo-type line search~\cite{MR191071}  that can be conducted in a nonmonotone fashion.
Specifically, at iteration $k$ it searches for an update that decreases the value of $\varphi$ with respect to the maximum function value attained among the last $\min\{\mM_k,k\}+1$ iterations, where $\mM_k\geq 0$ is a \emph{memory parameter} that is allowed to vary at each iteration.

\begin{algorithm}[ht!]
	\begin{algorithmic}[1]
		\Require{$x_0 \in \R^n$; $\tau_{\min} >0$; $\sigma,\beta\in{]0,1[}$.}
		\For{$k=0,1,\ldots$}
		\State Take $w_k \in  \partial \varphi(x_k)$; \textbf{if} $w_k=0$ \textbf{then} STOP and return $x_k$.\label{step:wk}
		\State Choose any $d_k \in\mathbb{R}^n\backslash\{0\}$ such that $\langle w_k,d_k\rangle<0$.  \label{step:dk}
		\State Take  any $\overline{\tau}_k \geq \tau_{\min}$. Set the stepsize $\tau_k:=\overline{\tau}_k$ and the memory parameter $\mM_{k}$. \label{step:mk}
		%\While{$\varphi(x_k + \tau_k d_k) \geq  \max\left\{ \varphi(x_i)  : {[k-\mM]^+\leq i \leq k} \right\}   +  \sigma \tau_k \langle w_k, d_k\rangle$}
		\While{$\varphi(x_k + \tau_k d_k) > \max\limits_{  {[k-\mM_k]^+\leq i \leq k}       }   \varphi(x_i)   +  \sigma \tau_k \langle w_k, d_k\rangle$}\label{step:while}
		\State $\tau_k = \beta \tau_k$;
		\EndWhile
		\State	  Set $x_{k+1}:=x_k + \tau_k d_k$.\label{step:update}
		\EndFor
	\end{algorithmic}
	\caption{Nonmonotone Subgradient Method}\label{alg:1}
\end{algorithm} %\vspace*{-0.05in}

In Assumption~\ref{a:Bk_bounded} below we specify some conditions that will be required for the convergence analysis of our scheme. We note that these are standard assumptions for Newton and quasi-Newton methods. Particularly,  they already appear in the seminal work\cite{MR849278}, where such a nonmonotone line search was introduced.

\begin{assumption}\label{a:Bk_bounded}
	The following conditions shall be assumed:
	\begin{enumerate}[label=(\alph*)]
		\item\label{ass:a}  There exists a constant $\Csta>0$ such that
		\begin{align}\label{assumpt01}
			\langle w_k,d_k\rangle \leq -\Csta \|d_k\|^2,\quad \text{ for all } k \in \N.
		\end{align}
		\item\label{ass:b} 	For any bounded  subsequence $(x_{k_j})_{j\in\N}$ generated by Algorithm~\ref{alg:1}, there is a constant  $\Cstb>0$  such that
		\begin{equation}\label{assumpt02}
			\| w_{k_j}\| \leq \Cstb \| d_{k_j}\|,   \quad \text{ for all } j \in \N.
		\end{equation}
		\item\label{ass:c} There exists  $\mM \geq 0$ such that the sequence of memory parameters satisfies that
		\begin{align}\label{assumpt03}
			0\leq \mM_{k+1} \leq \min\{ \mM_{k }+1, \mM \}, \quad \text{for  all }k \in \mathbb{N}.
		\end{align}
	\end{enumerate}
	
\end{assumption}

\begin{remark}[Ensuring Assumption~\ref{a:Bk_bounded}]\label{rem:eig}
	In many practical applications, the directions $d_k$ used in Algorithm \ref{alg:1} are computed as the solution to a linear system of the form
	\[
	-w_k = B_k d_k,
	\]
	where $B_k$ is a positive definite matrix. In this case, we highlight the interpretation of conditions \eqref{assumpt01} and \eqref{assumpt02} in terms of the spectral
	properties of~$B_k$:
	
	\begin{itemize}
		\item Inequality \eqref{assumpt01} amounts to requiring that there exists a positive constant $\lambda_{\texttt{m}}$ such that
		\[
		\lambda_{\texttt{m}} \leq \lambda_{\min}(B_k) , \quad \text{for all } k \in \mathbb{N}.
		\]
		
		\item Inequality \eqref{assumpt02}, on the other hand, is satisfied whenever for any bounded subsequence $(x_{k_j})_{j \in \mathbb{N}}$ generated by Algorithm~\ref{alg:1}, there exists a positive constant $\lambda_{\texttt{M}}$ such that
		\[
		\lambda_{\max}(B_{k_j}) \leq  \lambda_{\texttt{M}}, \quad \text{for all } j \in \mathbb{N}.
		\]
	\end{itemize}
	
	Lastly, condition \eqref{assumpt03} can be easily guaranteed by keeping a constant memory parameter. In Algorithm~\ref{alg:SNSM} we provide a different self-adaptive possibility.
\end{remark}

The \textbf{while} loop in step~\ref{step:while} of Algorithm~\ref{alg:1} performs a nonmonotone Armijo-like line search.
%The following proposition proves that, under our assumptions, such a line search is well-defined and ensures finite termination of the \textbf{while} loop.
When the function $\varphi$ is upper-$\mathcal{C}^2$, such a line search is well-defined and ensures finite termination of the \textbf{while} loop. This is a direct consequence of the following result.

\begin{proposition}\label{p:dm}
	Let $\varphi:\R^n\to\R$ be upper-$\mathcal{C}^2$ on  $\R^n$. %, and suppose that Assumption~\ref{a:Bk_bounded}\ref{ass:a}  holds.
	Given any $\sigma\in{]0,1[}$ and $x,w,d\in\R^n$ such that $w\in\partial\varphi(x)$ and $\langle w,d\rangle<0$, there exists $\eta>0$ such that
	\begin{align*}
		\varphi(x + \tau d) &\leq   \varphi(x)  +    \sigma \tau  \langle w,d\rangle, \quad\text{for all } \tau \in{]0,\eta[}.%\langle w_k, d_k\rangle \\&\leq\max_{   {[k-\mM_k]^+\leq i \leq k }} \varphi(x_i)  - \Csta  \sigma  \tau_k \|d_k\|^2.
	\end{align*}
\end{proposition}

%\begin{proposition}\label{p:dm}
%	Let $\varphi:\R^n\to\R$ be upper-$\mathcal{C}^2$ on  $\R^n$. %, and suppose that Assumption~\ref{a:Bk_bounded}\ref{ass:a}  holds.
%	Given any $\sigma\in{]0,1[}$, we have that for every $k \in \N$  there exists $\eta_k>0$ such that
%	\begin{align*}
	%		\varphi(x_k + \tau d_k) &\leq   \max_{   {[k-\mM_k]^+\leq i \leq k }}\varphi(x_i)  +    \sigma \tau  \langle w_k,d_k\rangle, \quad\text{for all } \tau \in{]0,\eta_k[}.%\langle w_k, d_k\rangle \\&\leq\max_{   {[k-\mM_k]^+\leq i \leq k }} \varphi(x_i)  - \Csta  \sigma  \tau_k \|d_k\|^2.
	%	\end{align*}	
%\end{proposition}

\begin{proof}
	Proposition~\ref{p:deslemma}\ref{it:p25-iii} yields the following estimate for the upper Dini directional derivative
	\begin{align*}
		d^{+}\varphi(x;d) &= \limsup_{t\to 0^+}  \frac{ \varphi(x+td)-\varphi(x)}{t}  \\
		& \leq  \limsup_{t\to 0^+} \frac{t\langle w,d \rangle+   \kappa t^2  \|d\|^2}{t} \\
		& = \langle w,d \rangle,
	\end{align*}
	%	where  $\kappa$ is some nonnegative constant.  By \eqref{assumpt01}, we have
	%	\begin{equation}\label{eq:bound_inner}
		%		d^{+} \varphi(x_k;d_k) \leq \langle w_k,d_k\rangle \leq -\Csta \|d_k\|^2.
		%	\end{equation}
	%	Then, since $d_k\neq 0$ and $\Csta>0$, we have that $ \langle w_k,d_k\rangle <0$, so there exists $\eta_k>0$ such that, for all $\tau\in{]0,\eta_k[}$,
	where  $\kappa$ is some nonnegative constant.  Then, we have
	\begin{equation}\label{eq:bound_inner}
		d^{+} \varphi(x;d) \leq \langle w,d\rangle<\sigma\langle w,d\rangle<0,
	\end{equation}
	so there exists $\eta>0$ such that, for all $\tau\in{]0,\eta[}$,
	\begin{equation*}
		\begin{aligned}
			\varphi(x+\tau d) &\leq \varphi(x) + \sigma\tau \langle w,d\rangle,% \\
			%			&\leq \max\left\{ \varphi(x_i)  : {[k-\mM_k]^+\leq i \leq k} \right\} + \sigma\tau \langle w_k,d_k\rangle,
		\end{aligned}
	\end{equation*}
	which completes the proof.
\end{proof}

Our next proposition proves the convergence of a specific subsequence of function values generated by Algorithm~\ref{alg:1}. The following notation shall be used. For all $k\in\N$, we use  $\ell(k)$ to denote the smallest integer such that
\begin{align}
	\ell(k) \in  \argmax\left\{ \varphi(x_i)  : {[k-\mM_k]^+\leq i \leq k} \right\}.
\end{align}

\begin{proposition} \label{prop01}
	Let $\varphi:\R^n\to\R$ be  upper-$\mathcal{C}^2$ on $\R^n$ with $\inf_{x\in\R^n}\varphi(x)>-\infty$ and suppose that Assumptions~\ref{a:Bk_bounded}\ref{ass:a} and~\ref{a:Bk_bounded}\ref{ass:c} hold. Given any initial point $x_0\in\R^n$, either Algorithm~\ref{alg:1}  stops at some stationary point after a finite number of iterations, or it produces a sequence $(x_{k})_{k\in\N}$ such that
	\begin{equation}\label{eq:suff_dec}
		\varphi(x_{k+1})   \leq    \varphi(x_{\ell(k)})  - \frac{\sigma \Csta }{ \tau_k  } \|x_{k+1 } - x_k\|^2, \quad \text{for all } k \geq 0.
	\end{equation}
	Particularly, the sequence $\bigl(\varphi(x_{\ell(k)})\bigr)_{k\in\N}$ monotonically decreases and converges. Furthermore, if $\sup_{k\in\N}\tau_k<+\infty$, there exists a constant $c >0$ such that
	\begin{align}\label{sum_diff}
		\min_{ 0 \leq j \leq k} \| x_{j+1} - x_{j}\| \leq  \frac{ c }{ \sqrt{k +1}}, \quad \text{for all } k \in \N. % \frac{\tau_{\max}  }{ \sigma \Csta  } \left( \varphi(x_0) - \inf_{ k \in \N }\varphi(x_k)  \right)
	\end{align}
\end{proposition}

\begin{proof}
	If Algorithm~\ref{alg:1} stops after a finite number of iterations due to the stopping criterion in step~\ref{step:wk}, then it clearly returns a stationary point of~\eqref{eq:P}. Otherwise, it generates an infinite sequence $(x_k)_{k\in\N}$. By step~\ref{step:while} of Algorithm~\ref{alg:1} and \eqref{assumpt01}, we have that for all $k \geq 0$,
	\begin{align*}
		\inf_{ x\in  \R^n}\varphi (x)\leq  \varphi(x_{k+1})    &\leq  \max\limits_{{[k-\mM_k]^+\leq i \leq k}}  \varphi(x_i) - \sigma \Csta \tau_k \|d_k\|^2 \\&=   \varphi(x_{\ell(k)})  -  \sigma \Csta \tau_k \|d_k\|^2 \\%=   \varphi(x_{\ell(k)})  -  \frac{ \sigma \Csta }{ \tau_k }\|x_{k+1}-x_k\|^2\\
		&\leq   \varphi(x_{\ell(k)})  -  \frac{ \sigma \Csta }{ \tau_k }\|x_{k+1}-x_k\|^2.
	\end{align*}
	It follows that \eqref{eq:suff_dec} holds.  Hence, using   \eqref{assumpt03}, we have that
	\begin{equation}\label{eq:monot}
		\varphi(x_{\ell(k+1)})\leq \varphi(x_{\ell(k)}),\quad\text{for all } k\geq0.
	\end{equation}
	Since $\bigl(\varphi(x_{\ell(k)})\bigr)_{k\in\N}$ is bounded from below, then it converges.
	
	Now, define $z(t):= \ell( t(\mM+1))$,  for $t \geq 0$. Since $\ell(t(\mM+1))\geq t(\mM+1)-\mM$, for $t\geq 1$, it holds that $z(t)-1\geq (t-1)(\mM+1)$, so we deduce by~\eqref{eq:monot} that
	$$\varphi(x_{\ell(z(t)-1)})\leq \varphi(x_{\ell((t-1)(\mM+1))})=
	\varphi(x_{z(t-1)}),\quad\text{for all }t\geq 1.$$
	Hence, by setting  $k:=z(t)-1$ for $t\geq 1$ in \eqref{eq:suff_dec},  we get that
	\begin{align*}
		\varphi(x_{ z(t)})  %& \leq    \varphi(x_{\ell( z(t))})  - \frac{\sigma \Csta }{ \tau_k  }\|x_{z(t)+1} - x_{ z(t)}\|^2\\
		& \leq  \varphi(x_{z(t-1)}) -\frac{\sigma \Csta }{ \tau_{z(t)-1}  } \|x_{z(t)} - x_{ z(t)-1}\|^2.
	\end{align*}
	
	If we let $\tau_{\max}:=\sup_{k\in\N}\tau_k$, we get for any integer $s\geq 1$ that
	\begin{align}\label{ineqsum}
		\begin{aligned}
			\sum_{t=1}^{s}  \|x_{z(t) } - x_{ z(t)-1}\|^2 &\leq   \frac{  \tau_{\max}   }{   \sigma \Csta }    \sum_{t=1}^{s} \left(  \varphi(x_{ z(t-1)  }) - \varphi(x_{z(t)})  \right) \\&=
			\frac{  \tau_{\max}   }{   \sigma \Csta }  \left(  \varphi( x_0) - \varphi(x_{ z( s) })   \right).
		\end{aligned}
	\end{align}
	Finally, let  $k \geq \mM$. By taking $s:= \bigl\lfloor \frac{k+1}{\mM+1} \bigr\rfloor$ in \eqref{ineqsum} (i.e., $s$ is the greatest integer less than or equal to $\frac{k+1}{\mM+1}$), and noticing that $z(s) \leq k+1$, we get that
	\begin{align*}
		\min_{ 0 \leq j \leq k} \| x_{j+1} - x_{j}\|^2 &\leq \frac{ 1}{ s}   	\sum_{t=1}^{s}  \|x_{z(t) } - x_{ z(t)-1}\|^2  \\
		%	& \leq  \frac{  \tau_{\max}  \left(  \varphi( x_0) - \varphi(x_{ z( s) })   \right)  }{   \sigma \Csta  } \frac{1}{ {\lfloor (k+1)/(M+1) \rfloor} }  \\
		& \leq \frac{\tau_{\max}}{ {\big\lfloor \frac{k+1}{\mM+1}\big\rfloor}\sigma \Csta }   \left(  \varphi( x_0) - \inf_{k\in \N}\varphi(x_{ \ell(k) })   \right),
	\end{align*}
	which yields \eqref{sum_diff}.	
	%Now,   by replacing $k$ by $\ell(k)-1$ in~\eqref{eq:suff_dec}, we obtain
	%\begin{equation*}
	%	\varphi(x_{\ell(k)}) \leq \varphi(x_{ \ell(\ell(k)-1) })  - \sigma \lambda^{-1}_{\max}(B_{\ell(k)-1}) \tau_{\ell(k)-1} \|d_{\ell(k)-1}\|^2.
	%\end{equation*}
	% It follows from the convergence of $\varphi(x_{\ell(k)})$ that $ \sigma \lambda^{-1}_{\max}(B_{\ell(k)-1}) \tau_{\ell(k)-1} \|d_{\ell(k)-1}\|^2 \to 0$ and  hence it completes the proof.
\end{proof}

We are now ready to present the main result of this paper. The following theorem establishes subsequential convergence of Algorithm~\ref{alg:1} to a stationary point of~\eqref{eq:P}.

%First, observe that condition~\eqref{assumpt01} is quite natural. In practical implementations, $ \lambd\Csta $ may be interpreted as the maximum capacity of a computer, and thus it does not present any challenges for applications. Moreover,
%Assumption~\ref{a:Bk_bounded} can be regarded as an analogous to the classical assumption on Quasi-Newton methods for the smooth convex case, where the norm Hessian of the objective function is assumed to be bounded from above and away from zero (see, e.g.,~\cite[Assumption~6.1]{MR2244940}). In \cite{MR2629829} the boundedness of the Hessian approximations has also been considered as a standing assumption for the convergence analysis of a \emph{modified} BFGS update for nonsmooth convex optimization (see also the technical report \cite{yu2010quasinewton}). However, we highlight that in our analysis we just need the matrices $B_k$ to be bounded according to a bounded subsequence of the whole iteration.

\begin{theorem}\label{t:conv}
	Let $\varphi:\R^n\to\R$ be upper-$\mathcal{C}^2$ on $\R^n$ with $\inf_{x\in\R^n}\varphi(x)>-\infty$ and suppose that Assumption~\ref{a:Bk_bounded} holds.  Given any initial point $x_0\in\R^n$, either Algorithm~\ref{alg:1}  stops at a stationary point after a finite number of iterations, or it produces a sequence $(x_{k})_{k\in\N}$ such that, if $\sup_{k\in\N}\tau_k<+\infty$, the following assertions hold:
	
	%Assume that the sequence generated by Algorithm~\ref{alg:1} contains a bounded subsequence $(x_{k_j})_{j\in\N}$. Then for every $p \in \N$, the following conditions hold:
	\begin{enumerate}[label=(\roman*)]
		\item\label{it:conv-1} If $(x_{k_j})_{j\in\N}$ is a  bounded subsequence  of $(x_k)_{k \in  \N}$, then  $\inf_{ j \in \N} \tau_{k_j} >0$. Moreover, $\lim\limits_{j \to \infty } \|x_{k_{j}+1}-x_{k_j}\|= 0$, $\lim\limits_{j \to \infty } \|w_{k_j}\|  =0$, $\lim\limits_{j \to \infty } \| d_{k_j}\| =0$ and   $\lim\limits_{j \to \infty} \varphi(x_{k_j}) = \lim\limits_{k \to \infty} \varphi(x_{\ell( k)})$.
		
		% for every for every $p \in \mathbb{Z}$ we have that  $\lim\limits_{j \to \infty } \|x_{k_{j+p+1}}-x_{k_j+p}\|\to 0$ and   $\lim\limits_{j \to \infty} \varphi(x_{k_j+p}) = \lim\limits_{k \to \infty} \varphi(x_{\ell( k)})$.
		
		\item\label{it:conv-2}  If the entire sequence $(x_k)_{k\in\N}$ is bounded, then  its set of accumulation points is nonempty, closed and connected. Moreover, in such a case, there exists $\Cstc>0$ such that
		\begin{align}\label{rates}
			\min_{ 0 \leq j \leq k} \left\{ \| x_{j+1} - x_{j}\|  + \| w_j\| + \| d_j\|\right\} \leq  \frac{ \Cstc}{ \sqrt{ k+1}}, \  \text{for all } k \geq 0.
		\end{align}
		
		\item\label{it:conv-3} Any accumulation point of $(x_k)_{k \in  \N}$ is a stationary point of~\eqref{eq:P}.
		
		\item\label{it:conv-4} If $(x_k)_{k\in\N}$ is bounded, and it has an isolated accumulation point, then the entire sequence converges to a stationary point of~\eqref{eq:P}.
	\end{enumerate}
	
\end{theorem}
\begin{proof}
	Let  $(x_{k_j})_{j\in \N }$ be a bounded subsequence of $(x_k)_{k\in\N}$. To simplify the notation, let us assume without loss of generality that $k_j > \mM$ for all $j  \in \N$. We divide the proof into the following  claims.\smallskip
	
	\noindent\textbf{Claim 1:} \textit{It holds that $\inf_{j \in \N} \tau_{k_j} >0$.}\smallskip
	
	We proceed by way of contradiction. Suppose there exists a subsequence $(\tau_{\nu_i})_{i \in \N}$ of $(\tau_{k_j})_{j \in \N}$ such that $\tau_{\nu_i} \to 0^+$ as $i \to \infty$. Given the boundedness of the sequence $(x_{k_j})_{j \in \N}$, we may, without loss of generality, assume that the corresponding subsequence $(x_{\nu_i})_{i \in \N}$ converges to some point $\bar{x}$. Moreover, since $\varphi$ is locally Lipschitz continuous and each $w_{\nu_i}$ belongs to $\partial\varphi(x_{\nu_i})$, it follows that the sequence $(w_{\nu_i})_{i \in \N}$ is bounded (see, e.g., \cite[Proposition~2.1.2]{MR1488695}). Hence, by possibly passing to a further subsequence without relabeling, we may assume that $w_{\nu_i} \to \bar{w}$ for some $\bar{w}\in\R^n$.

	Now, let us observe that assumption \eqref{assumpt01} leads to the inequality
	\begin{equation}\label{eq:wgeqd}
		- \langle w_{\nu_i}, d_{\nu_i} \rangle \geq \Csta \|d_{\nu_i}\|^2, \quad \text{for all } i \in \N.
	\end{equation}
	Applying the Cauchy--Schwarz inequality to the left-hand side of \eqref{eq:wgeqd}, and recalling that $d_{\nu_i} \neq 0$, we obtain the bound
	\begin{equation*}
		\|w_{\nu_i}\| \geq \Csta \|d_{\nu_i}\|, \quad \text{for all } i \in \N,
	\end{equation*}
	which implies that the sequence $(d_{\nu_i})_{i \in \mathbb{N}}$ is bounded. Therefore, by passing to a subsequence if necessary, we may assume that $(d_{\nu_i})_{i\in\N}$ converges to some point $\bar{d} \in \R^n$.
	
	Since $\tau_{\nu_i} \to 0^+$, we can assume that $\tau_{\nu_i} < \tau_{\min}$ for all $i \in \N$. Then, by step~\ref{step:while} of Algorithm~\ref{alg:1}, we have
	\begin{equation}\label{eq:cont-1}
		\varphi(x_{\nu_i}+ \beta^{-1} \tau_{\nu_i}d_{\nu_i}) > \max_{ [\nu_i-\mM_{\nu_i}]^+ \leq s \leq \nu_i } \varphi(x_{s}) + \sigma \beta^{-1}\tau_{\nu_i} \langle w_{\nu_i},d_{\nu_i} \rangle.
	\end{equation}
	Note that $x_{\nu_i} + \beta^{-1} \tau_{\nu_i} d_{\nu_i} \to \bar{x}$. Thus, by applying Proposition~\ref{p:deslemma}\ref{it:p25-iii} at the point $\bar{x}$, there exists a constant $\kappa \geq 0$ such that, for sufficiently large $i \in \N$,
	\begin{equation}\label{eq:cont-2}
		\varphi(x_{\nu_i}+ \beta^{-1} \tau_{\nu_i}d_{\nu_i}) \leq \varphi(x_{\nu_i}) + \beta^{-1}\tau_{\nu_i} \langle w_{\nu_i},d_{\nu_i} \rangle + \kappa \beta^{-2} \tau_{\nu_i}^2 \|d_{\nu_i}\|^2.
	\end{equation}
	Combining inequalities \eqref{eq:wgeqd}--\eqref{eq:cont-2}, we deduce that
	\[
	\begin{aligned}
		\sigma \beta^{-1} \tau_{\nu_i} \langle w_{\nu_i}, d_{\nu_i} \rangle &< \varphi(x_{\nu_i} + \beta^{-1} \tau_{\nu_i} d_{\nu_i}) - \max_{ [\nu_i - \mM_{\nu_i}]^+ \leq s \leq \nu_i } \varphi(x_s) \\
		&\leq \varphi(x_{\nu_i} + \beta^{-1} \tau_{\nu_i} d_{\nu_i}) - \varphi(x_{\nu_i}) \\
		&\leq \beta^{-1} \tau_{\nu_i} \langle w_{\nu_i}, d_{\nu_i} \rangle + \kappa \beta^{-2} \tau_{\nu_i}^2 \|d_{\nu_i}\|^2 \\
		& \leq \beta^{-1} \tau_{\nu_i} \left(1 - \frac{\kappa}{\Csta \beta} \tau_{\nu_i} \right) \langle w_{\nu_i}, d_{\nu_i} \rangle,
	\end{aligned}
	\]
	for all sufficiently large $i \in \N$. Rearranging the resulting expression, and using the fact that $\langle w_{\nu_i}, d_{\nu_i} \rangle < 0$, we conclude that
	\[
	\sigma > 1 - \frac{\kappa}{\Csta \beta} \tau_{\nu_i}.
	\]
	Since $\tau_{\nu_i} \to 0^+$, this inequality contradicts the assumption that $\sigma \in (0, 1)$. Hence, this ends the proof of the claim.\smallskip
	
	\textbf{Claim 2:} \textit{For every $p \in \{0,\ldots,\mM+1\}$, the sequences $(x_{k_j+p})_{j \in \mathbb{N}}$ and $(x_{k_j-p})_{j \in \mathbb{N}}$ are bounded.}
	
	\smallskip
	
	We prove this claim by induction. The base case $p = 0$ holds by assumption. Now, suppose that $(x_{k_j+p})_{j \in \mathbb{N}}$ and $(x_{k_j-p})_{j \in \mathbb{N}}$ are bounded for some $p \in \{0,\ldots,\mM\}$.  Let us define
	\begin{equation}\label{def:hatsigma}
		\hat{\sigma} := \frac{\sigma \Csta}{\tau_{\max}}.
	\end{equation}
	Using the sufficient decrease condition \eqref{eq:suff_dec}, we obtain that
	\begin{align*}
		\hat{\sigma} \|x_{k_j+p+1} - x_{k_j+p}\|^2 & \leq \frac{\sigma \Csta}{\tau_{k_j+p}} \|x_{k_j+p+1} - x_{k_j+p}\|^2 \\
		%&= \frac{\sigma \Csta}{\tau_{\max}} \|x_{k_j+p+1} - x_{k_j+p}\|^2 \leq \frac{\sigma \Csta}{\tau_{k_j+p}} \|x_{k_j+p+1} - x_{k_j+p}\|^2 \\
		&\leq \varphi(x_{\ell(k_j+p)}) - \varphi(x_{k_j+p+1}) \\
		&\leq \varphi(x_{0}) - \inf_{x \in \mathbb{R}^n} \varphi(x),
	\end{align*}
	where the last inequality follows from the fact that $\varphi(x_{\ell(k)}) \leq \varphi(x_{\ell(0)})=\varphi(x_{0})$ for all $k$ (recall~\eqref{eq:monot}). Since  $(x_{k_j+p})_{j \in \mathbb{N}}$  is bounded, this upper bound implies that $(x_{k_j+p+1})_{j \in \mathbb{N}}$ is also bounded. A similar argument, applying~\eqref{eq:suff_dec} to ${k_j-p-1}$, permits to show that $(x_{k_j-p-1})_{j \in \mathbb{N}}$ is also bounded, and completes the inductive step. \smallskip
	
	%Indeed,
	%
	%\begin{align*}
	%	\frac{\sigma}{ \tau_{\max}  \lambd\Csta }  \|x_{k_j-p } - x_{k_j-p-1}\|^2  &\leq \frac{\sigma}{ \tau_{k_j-p-1} \lambda_{\max}(B_{k_j-p-1}) } \|x_{k_j-p } - x_{k_j-p-1}\|^2  \\
	%	&\leq  \varphi(x_{\ell(k_j-p-1)})    - \varphi(x_{ k_j-p})   \\&\leq \varphi(x_{\ell(0)}) - \inf_{x\in \R^n} \varphi (x).
	%\end{align*}
	
	\textbf{Claim 3:} \textit{For every $p \in \{0,\ldots,\mM+1\}$, the sequence $(x_{\ell(k_j+m+1)-p})_{j \in \mathbb{N}}$ is bounded.}\smallskip
	
	Let us notice that
	$$k_j- \mM\leq k_j+1- p\leq \ell(k_j+m+1) - p \leq k_j+m+1- p\leq k_j+\mM+1.$$ Then,
	\[
	x_{\ell(k_j+m+1)-p} \in \bigcup_{i = -\mM}^{\mM+1} \{ x_{k_j+ i} \},\quad\text{for all }j\in\N,
	\]
	which shows that $(x_{\ell(k_j+m+1)-p})_{j \in \mathbb{N}}$ is bounded, due to Claim 2.\smallskip
	
	\textbf{Claim 4:} \textit{For every $p\in\{0,\ldots,\mM\}$, we have that}
	\begin{align}
		&\lim_{j \to \infty} \| x_{\ell(k_j+\mM+1)-p} - x_{\ell(k_j+\mM+1) - p-1} \| = 0 \label{eq:limitDX}, \\
		&\lim_{j \to \infty} \varphi(x_{\ell(k_j+\mM+1) - p-1}) = \lim_{k \to \infty} \varphi(x_{\ell(k)}) \label{eq:limitvarphiX}.
	\end{align}
	%	First, let us notice that, by Claim 3, for every $p \in \{0,\ldots,\mM+1\}$, the sequence $(x_{\ell(k_j+\mM+1) - p})_{j \in \N}$ is bounded.
	
	We proceed by induction on $p \in\{0,\ldots,\mM\}$. First, we prove both~\eqref{eq:limitDX} and~\eqref{eq:limitvarphiX} for $p = 0$.  By replacing $k$ by $\ell(k_j+\mM+1)-1$ in~\eqref{eq:suff_dec}, we obtain
	\begin{align*}
		\varphi(x_{\ell(k_j+\mM+1)}) &\leq \varphi(x_{\ell(\ell(k_j+\mM+1)-1)}) - \hat{\sigma} \|x_{\ell(k_j+\mM+1)} - x_{\ell(k_j+\mM+1)-1}\|^2,
	\end{align*}
	where $\hat{\sigma}$ was defined in \eqref{def:hatsigma}. Recalling that $\lim_{k \to \infty} \varphi(x_{\ell(k)})$ exists, we get that~\eqref{eq:limitDX} holds for $p = 0$.
	
	Now, to prove \eqref{eq:limitvarphiX}, observe that since $\varphi$ is locally Lipschitz, it is uniformly continuous over the bounded set $\{ x_{\ell(k_j+\mM+1)}, x_{\ell(k_j+\mM+1)-1} : j \in \mathbb{N} \}$ (recall Claim~3), which implies by~\eqref{eq:limitDX} that
	\[
	\lim_{j \to \infty} \varphi(x_{\ell(k_j+\mM+1)-1}) = \lim_{j \to \infty} \varphi(x_{\ell(k_j+\mM+1)}) = \lim_{k \to \infty} \varphi(x_{\ell(k)}).
	\]
	This proves the claim for $p = 0$.
	
	Now, assume that \eqref{eq:limitDX} and \eqref{eq:limitvarphiX} hold for some $p \in \{0,\ldots,\mM-1\}$. Similarly, by replacing $k$ by $\ell(k_j+\mM+1)-p-2$ in~\eqref{eq:suff_dec}, we obtain
	\begin{align*}
		\varphi(&x_{\ell(k_j+\mM+1)-p-1})\\
		&\leq \varphi(x_{\ell(\ell(k_j+\mM+1)-p-2)}) - \hat{\sigma} \|x_{\ell(k_j+\mM+1)-p-1} - x_{\ell(k_j+\mM+1)-p-2}\|^2.
	\end{align*}
	This implies that \eqref{eq:limitDX} holds for $p+1$, since, by the induction hypothesis, 
	\[
	\lim_{j \to \infty} \varphi(x_{\ell(k_j+\mM+1)-p-1}) = \lim_{k \to \infty} \varphi(x_{\ell(k)}).
	\]
	Finally, by uniform continuity of $\varphi$ over the bounded set $\{ x_{\ell(k_j+\mM+1)-p-1}, x_{\ell(k_j+\mM+1)-p-2} : j \in \mathbb{N} \}$ (recall Claim 3), we get that \eqref{eq:limitvarphiX} holds for $p+1$. This completes the proof of the claim.\smallskip

	\textbf{Claim 5:} \textit{Assertion~\ref{it:conv-1} holds.} \smallskip
	
	Indeed, the first part of assertion~\ref{it:conv-1} was proved in Claim~1. Now, let us observe that $k_j+1\leq \ell(k_j + \mM+1) \leq k_j + \mM+1$, so $k_j+1=\ell(k_j+\mM+1)-p$ for some $p\in\{0,\ldots,\mM\}$. Therefore, using Claim~4, we obtain
	\begin{align*}
		\| x_{k_j + 1} - x_{k_j} \| &\leq \max_{p = 0, \ldots, \mM} \| x_{\ell(k_j+\mM+1) -p} - x_{\ell(k_j+\mM+1) - p-1} \| \to 0, \\
		| \varphi(x_{k_j}) - \varphi^\ast | &\leq \max_{p = 0, \ldots, \mM} | \varphi(x_{\ell(k_j+m+1)-p-1}) - \varphi^\ast | \to 0,
	\end{align*}
	where $\varphi^\ast := \lim_{k \to \infty} \varphi(x_{\ell(k)})$. Now, recalling Claim~1 and the fact that
	\[
	\| d_{k_j} \| = \frac{ 1}{ \tau_{k_j} } \| x_{k_j+1} - x_{k_j}\| \leq  \frac{1}{\inf_{ j \in \N} \tau_{k_j} }\| x_{k_j + 1} - x_{k_j} \|,
	\]
	we conclude that $\lim_{j \to \infty} \| d_{k_j} \| = 0$. Finally, by the upper bound~\eqref{assumpt02}, we have that $\lim_{j \to \infty} \| w_{k_j} \| = 0$. \smallskip

	\textbf{Claim 6:} \textit{Assertion~\ref{it:conv-2}  holds.} \smallskip
	
	Indeed, if $(x_k)_{k \in \mathbb{N}}$ is bounded,  Claim 1  yields
	\begin{align}\label{ineqdk}
		\| d_k\| &\leq \frac{1}{\inf_{ j \in \N} \tau_j }\| x_{k+1} - x_k\|, \quad \text{for all  }k \in \N.
	\end{align}
	Therefore, applying \eqref{assumpt02}, \eqref{ineqdk} and \eqref{sum_diff}  we get that
	$$\min_{ 0 \leq j \leq k} \left\{ \| x_{j+1} - x_{j}\|  + \| w_j\| + \| d_j\|\right\} \leq  \left(1+\frac{\Cstb+1}{\inf_{ j \in \N} \tau_j }\right)\frac{ \Cstc}{ \sqrt{ k+1}},$$
	so \eqref{rates} holds.
	
	Moreover, by assertion~\ref{it:conv-1}, we have that the sequence $(x_k)_{k \in \mathbb{N}}$ satisfies the so-called \emph{Ostrowski condition}, that is,
	\[
	\lim_{k \to \infty} \| x_{k+1} - x_k \| = 0,
	\]
	which implies that the set of accumulation points of $(x_k)_{k \in \mathbb{N}}$ is nonempty, closed, and connected (see, e.g., \cite[Theorem~8.3.9]{MR1955649}).\smallskip
	
	\textbf{Claim 7:} \textit{Assertions~\ref{it:conv-3} and~\ref{it:conv-4} hold.} \smallskip
	
	Suppose that $x_{k_j}\to\bar{x}$ when $j\to\infty$. Using assertion~\ref{it:conv-1}, we have that $w_{k_j} \to 0$ as $j \to \infty$. Since $w_{k_j} \in \partial \varphi(x_{k_j})$ for all $j \in \mathbb{N}$, the closedness of the subdifferential mapping of  locally Lipschitz functions (see, e.g., \cite[Proposition~2.1.5]{MR1488695}) implies that $0 \in \partial \varphi(\bar{x})$.
	
	Assertion~\ref{it:conv-4} follows from \cite[Proposition~8.3.10]{MR1955649}, which ensures that the whole sequence $(x_k)_{k \in \mathbb{N}}$ converges to a point $\bar{x}$.
	This point $\bar{x}$ is stationary due to assertion~\ref{it:conv-3}.
\end{proof}

\begin{remark}
	(i)	Although Algorithm~\ref{alg:1} and Theorem~\ref{t:conv} are stated in terms of the Clarke subdifferential, we emphasize that the same conclusions of Theorem~\ref{t:conv} remain valid when the sequence $(w_k)_{k \in \N} $ is selected from the (Mordukhovich) limiting subdifferential of $\varphi$. This follows from the fact that the limiting subdifferential is contained within the Clarke subdifferential, ensuring that the descent condition \eqref{equpperdesc} still holds. In this case, the convergence result guarantees stationarity in the sense of the Mordukhovich subdifferential, which is a stronger condition than the Clarke-type stationarity defined in \eqref{eq:st-Clarke}.
	
	More generally, a closer examination of the proof of Theorem~\ref{t:conv}\ref{it:conv-3} reveals that one can establish $\lim_{j \to \infty} w_{k_j} = 0$, provided that the subsequence $(x_{k_j})_{j \in \mathbb{N}}$ converges. This, in turn, allows the interpretation of stationarity to vary depending on the specific subdifferential from which the subgradients $w_{k_j}$ are taken, potentially yielding different notions of stationary points.

	(ii) While  assumptions of the form of $\sup_{k\in\N} \tau_k < +\infty$, as required in Theorem~\ref{t:conv}, are common in the literature (see, e.g.,~\cite[Theorem~3.11]{aragonartacho2023boosted}),
	this condition can also be enforced by restricting the trial stepsize $\bar{\tau}_k$ to be bounded from above by a positive constant  in step~\ref{step:mk} of Algorithm~\ref{alg:1}.
	This approach was adopted, for instance, in~\cite[Algorithm~1]{aragonartacho2023coderivativebased}. Notably, imposing this upper bound was unnecessary in our numerical experiments,
	since the sequence remained bounded even under the self-adaptive rule introduced in the next section.
\end{remark}

\section{Self-adaptive setting of stepsize and memory parameters}\label{sect:sa}

The possibility of performing nonmonotone iterations might confront the purpose of the line search, which aims to achieve a significant reduction of the objective value. An inadequate choice of the memory parameter $m_k$ or the initial trial stepsize parameter $\bar{\tau}_k$ could slow down the resulting algorithm. To obtain a good performance, both parameters need to be chosen in accordance.
The purpose of this section is to propose a specification of Algorithm~\ref{alg:1} in which both $\bar{\tau}_k$ and $m_k$ are automatically updated by the scheme. Its pseudocode is presented in Algorithm~\ref{alg:SNSM} and its principal features  are discussed below.

\begin{algorithm}[h]
	\begin{algorithmic}[1]
		\Require{$x_0 \in \R^n$; $\tau_{\min} >0$; $\overline{\tau}_0\geq\tau_{\min}$; $\sigma,\beta\in{]0,1[}$; $\gamma>1$; $\mM_0,\mM\in\N$ with $\mM_0\leq\mM$.}
		\State{Set $\bar{\tau}_{-1}:=\overline{\tau}_0$ and $\tau_{-1}:=\overline{\tau}_0$.}
		\For{$k=0,1,\ldots$}
		\State Take $w_k \in  \partial \varphi(x_k)$; \textbf{if} $w_k=0$ \textbf{then} STOP and return $x_k$.\label{step2:wk}
		\State Choose any $d_k \in\mathbb{R}^n\backslash\{0\}$ such that $\langle w_k,d_k\rangle<0$.  \label{step2:dk}
		\State Set the stepsize $\tau_k:=\overline{\tau}_k$. \label{step2:initial_stepsize}
		\State{\textbf{if} $\varphi(x_k + \tau_k d_k) > \max\limits_{  {[k-\mM_k]^+\leq i \leq k}       }   \varphi(x_i)   +  \sigma \tau_k \langle w_k, d_k\rangle$ \textbf{then}}\label{step2:mk}
		\State{\hspace{\algorithmicindent}  $\mM_k = \min\{\mM_k + 1,\mM\}$;}
		\State{\textbf{end if}}\label{step2:endmk}
		\While{$\varphi(x_k + \tau_k d_k) >  \max\limits_{  {[k-\mM_k]^+\leq i \leq k}       }   \varphi(x_i)   +  \sigma \tau_k \langle w_k, d_k\rangle$}\label{step2:while}
		\State $\tau_k = \beta \tau_k$;
		\EndWhile
		
		\State	  Set $x_{k+1}:=x_k + \tau_k d_k$.\label{step2:update}
		\If{$\tau_{k-1}=\overline{\tau}_{k-1}$ \textbf{and} $\tau_{k}=\overline{\tau}_{k}$}
		\State $\overline{\tau}_{k+1}:=\gamma\tau_{k}$;
		\State $\mM_{k+1} := 0$;
		\Else
		\State $\overline{\tau}_{k+1}:=\max\{\tau_{k},\tau_{\min}\}$;
		%		\State $\mM_{k+1} = \mM_k$;
		\State $\mM_{k+1} := \min\left\{j \in \{0,\dots,\mM_k\} : \varphi(x_k + \tau_k d_k) \leq \varphi(x_{k-j}) + \sigma \tau_k \langle w_k, d_k\rangle \right\}$;
		
		\EndIf
		\EndFor
	\end{algorithmic}
	\caption{Self-adaptive Nonmonotone Subgradient Method (\salg)}\label{alg:SNSM}
\end{algorithm} %\vspace*{-0.05in}

Observe that the main computational steps  of Algorithm~\ref{alg:SNSM} coincide with those of Algorithm~\ref{alg:1}: steps~\ref{step2:wk}-\ref{step2:dk} and steps~\ref{step2:while}-\ref{step2:update} are also part of Algorithm~\ref{alg:1}. The most significant difference is that Algorithm~\ref{alg:SNSM} automatically selects the initial trial stepsize~$\bar{\tau}_k$ and the memory parameter~$m_k$ based on the performance of the line search in the previous iterations. First, the initial parameter for the line search is set in step~\ref{step2:initial_stepsize} as $\tau_k = \bar{\tau}_k$.
Then, steps~\ref{step2:mk}-\ref{step2:endmk} check whether the decrease condition is satisfied by the current memory parameter~$m_k$, otherwise $m_k$ is increased by one (unless it gets larger than the maximum memory parameter~$m$). In this way, these steps relax the acceptance condition by exploiting the nonmonotonicity with the purpose of permitting larger stepsizes.

After performing the nonmonotone  Armijo-type line search and updating the iterate in step~\ref{step2:update}, the algorithm determines the initial trial stepsize $\bar{\tau}_{k+1}$ and the memory parameter $m_{k+1}$ for the next iteration. If at the previous and current iterations the initial trial stepsizes were accepted (i.e., $\tau_{k-1} = \bar{\tau}_{k-1}$ and $\tau_{k}=\bar{\tau}_k$), then  the starting stepsize for iteration $k+1$ is increased by setting $\bar{\tau}_{k+1}:=\gamma \tau_k$, with $\gamma>1$. In addition, the memory parameter is set to~$m_{k+1}:=0$, thereby initially aiming for a monotone decrease at the next iteration. Otherwise, if the initial trial stepsize failed to be accepted in any of the two previous iterations, we take $\bar{\tau}_{k+1} := \max\{\tau_k,\tau_{\min}\}$ and the memory parameter $m_{k+1}$ is set to the smallest value satisfying the nonmonotone decrease condition, this is
\[
\mM_{k+1} := \min\left\{j \in \{0,\dots,\mM_k\} : \varphi(x_k + \tau_k d_k) \leq \varphi(x_{k-j}) + \sigma \tau_k \langle w_k, d_k\rangle \right\}.
\]

\begin{remark}
	Note that if one chooses $m=0$, then $m_k=0$ for all $k\in\N$ and Algorithm~\ref{alg:SNSM} becomes a monotone method. In such a case, the self-adaptive procedure is limited to determining the trial initial stepsize $\bar{\tau}_k$ at every iteration. This is exactly the routine described in~\cite[Algorithm~3]{aragonartacho2023coderivativebased}. This and similar procedures have been reported in numerical experiments to significantly accelerate the convergence of other line search algorithms (see, e.g.,~\cite{aragonartacho2023boosted,MR4078808}).
\end{remark}

\section{Minimum sum-of-squares clustering problem}
\label{sect:appl}

In this section, we present an application of \salg\ (Algorithm~\ref{alg:SNSM}) to  the \emph{minimum sum-of-squares clustering problem}, which is an unsupervised machine learning technique. In what follows, we devise a concrete implementation of the algorithm to this problem by providing a formula for the computation of subgradients of the objective function and appropriate search directions that perform well in practice, as illustrated in the numerical experiments reported in Section~\ref{sect:numerics}.

This data mining task aims to split a set of data points $\{a^1,\ldots,a^p\}\subseteq\R^s$ into $\ell$ groups or \emph{clusters} by minimizing the squared distance of each point to a representative of its group, called \emph{centroid}. Such a problem can be expressed as the minimization of the function ${\varphi}_{MSC}:\R^{s\times \ell} \to\R$ defined as
\begin{equation}\label{eq:f_clust}
	\varphi_{MSC}(X) := \frac{1}{p} \sum_{j=1}^p  \omega^{j}(X), \quad \text{with } \omega^j(X):= \min_{t\in\{1,2,\ldots,\ell\}}\|x^t - a^j\|^2,
\end{equation}
where $X=(x^1, x^2,\ldots,x^{\ell}) \in\R^{s\times\ell}$ contains the $\ell$ centroids to be found. Observe that  $\varphi_{MSC}$ perfectly fits under our framework, as it is upper-$\mathcal{C}^2$ but nonconvex  and nondifferentiable (see Figure~\ref{fig:Example_MSC}). Indeed,
the functions $\omega^j$ are upper-$\mathcal{C}^2$  by definition, so $\varphi_{MSC}$ is upper-$\mathcal{C}^2$.

\begin{figure}[h!]\centering
	\includegraphics[width=0.55\textwidth]{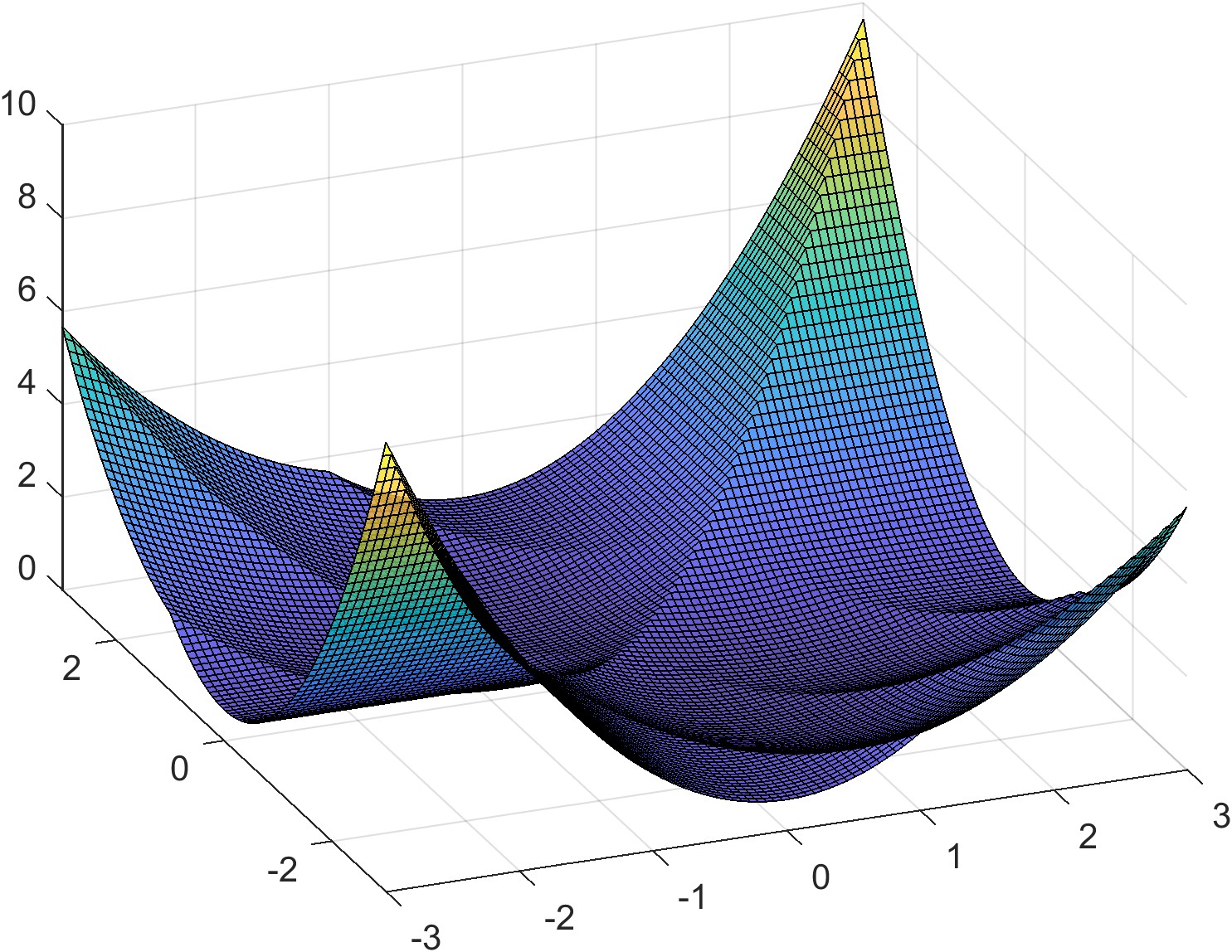}
	\caption{Illustration of the function $\varphi_{MSC}$ for $p=3$, $s=1$ and $\ell=2$, with $(a^1,a^2,a^3)=(-1,0,1)$}\label{fig:Example_MSC}
\end{figure}

The essential steps for the implementation of \salg\ for this problem are the selection of a Clarke subgradient of $\varphi_{MSC}$ in step~\ref{step2:wk} and  the computation of an appropriate descent direction in step~\ref{step2:dk}. The former has been discussed in \cite[Section~4.2]{aragonartacho2023boosted}, where the authors provided a close expression for an element in the limiting subdifferential of the functions $\omega^j$.  Accordingly, we propose in step~\ref{step2:wk} of the $k$-th iteration to choose an element of the form
\begin{equation}\label{eq:subgradient}
	w_k = \frac{1}{p} \sum_{j=1}^p w_k^j,
\end{equation}
with%
\begin{equation}\label{eq:subgradient_j}
	w_k^j \in  \left\{    (0,\ldots, 0, \underbrace{2(x^t_k - a^j)}_{
		t\text{-th position} }, 0 , \ldots,0)\in\R^{s\times\ell}     : \omega^j(X_k) = \| x^t_k - a^j\|^2         \right\},
\end{equation}
where $X_k:=(x_k^1,\ldots,x_k^{\ell})$.  Note that, in general, $w_k$ is not guaranteed to belong to the limiting subdifferential because of the possible failure of the subdifferential sum rule, but $w_k\in\partial\varphi_{MSC}(X_k)$ (see~\cite[Proposition~4.1]{aragonartacho2023boosted}). This choice allows for an efficient subgradient evaluation while leads to an adequate optimality condition.

Secondly, in order to compute the search direction $d_k$ for step~\ref{step2:dk}, we propose to incorporate second-order information by employing the Hessian matrix of a local approximation to $\varphi_{MSC}$ that only takes into consideration one of the~active indices. More precisely, at the $k$-th iteration we select, for each $j\in\{1,\ldots,p\}$, an active index $i_k(j)$ at $X_k$, i.e.,
\begin{equation}\label{activeindex}
	i_k(j)\in \argmin_{t\in\{1,\dots,\ell\}} \|x_k^t - a^j\|^2.
\end{equation}
For such indices,  we define the function
\[
{\varphi}^k_{MSC}(X) := \frac{1}{p} \sum_{j=1}^p \|x^{i_k(j)} - a^j\|^2.
\]
Since $i_k(j)$ is fixed for all $j\in\{1,\ldots,p\}$, this function is just a sum of squared norms and hence differentiable. Moreover, its value at the current iterate coincides with that of the original function; i.e., ${\varphi}^k_{MSC}({X_k})={\varphi}_{MSC}({X_k})$.
Now, note that the Hessian of ${\varphi}^k_{MSC}$ corresponds to a diagonal matrix with diagonal given block-wise by twice the number of times that each block variable appears in the sum, i.e.,
\[
\nabla^2 {\varphi}^k_{MSC}(X)  = \frac{1}{p}\blkdiag\bigl(2 q_{1}(k) I,2 q_{2}(k) I,\ldots,2 q_{\ell}(k) I\bigr),
\]
where $I$ is the identity matrix in $\R^{s\times s}$ and $q_{t}(k)$ denotes the number of times that each index $t\in\{1,2,\ldots,\ell\}$  has been chosen as an active index in  \eqref{activeindex}; that is,
\[q_t(k)=\left|\left\{ j\in\{1,\ldots,p\} : i_k(j)=t\right\}\right|.\]
Observe that $\nabla^2{\varphi}^k_{MSC}$ may be singular at some point if some of the indices are not active for any $j\in\{1,2,\ldots,p\}$. To overcome this problem,  we add a regularization term $\alpha_k>0$, and select $d_k$ in step~\ref{step2:dk} of \salg\ as
\begin{equation}\label{eq:dkmsos}
	d_k = - \left( \nabla^2 {\varphi}^k_{MSC}(X_k) +\alpha_k I \right)^{-1} w_k.
\end{equation}
Observe that the diagonal form of the above Hessian allows for an efficient numerical computation of the descent direction without the need to construct the whole matrix or
performing full matrix inverse operations. Indeed, one can compute $d_k\in\R^{s\times\ell}$ as %
\[
d_k = -\left( \left(\frac{p}{ 2 q_{1}(k)+{p}\alpha_k}, \frac{p}{2 q_{2}(k)+{p}\alpha_k}, \ldots, \frac{p}{ 2 q_{\ell}(k)+{p}\alpha_k} \right) \otimes \mathbbm{1}_s\right) \odot w_k ,\]%
where $\otimes$ stands for the Kronecker product, $\odot$ denotes the Hadamard product and $\mathbbm{1}_s$ is the $s$-dimensional column vector  of all ones.
Finally, observe that if one chooses $\alpha_k\in[\alpha_{\min},\alpha_{\max}]$ for all $k\in\N$, with $0<\alpha_{\min}<\alpha_{\max}$, then Assumptions~\ref{a:Bk_bounded}\ref{ass:a}-\ref{ass:b} hold, by Remark~\ref{rem:eig}. This is particularly true if one takes $\alpha_k=\alpha>0$, for all $k\in\N$.

\section{Numerical experiments}\label{sect:numerics}

In the following, we present some computational tests where we compare the performance of \salg\ against other commonly used algorithms. Specifically, in Subsection~\ref{subsec:SOS} we test the \emph{Difference of Convex functions Algorithm} (\emph{DCA}) introduced in~\cite{TAO1986249},  the \emph{inertial Difference of Convex functions Algorithm} (\emph{iDCA}) presented in~\cite{MR4027663} (which includes an extrapolation term to the classical DCA), the \emph{Boosted Difference of Convex function Algorithm} (\emph{BDCA}) proposed in~\cite{AragonArtacho2018,MR4078808} (which incorporates a line search to  DCA), as well as the \emph{Regularized Code\-ri\-va\-tive-based damped Semi-Newton method} (\emph{RCSN}) from~\cite{aragonartacho2023coderivativebased}. In addition, in Subsection~\ref{subsec:kmeans}, we evaluate the performance of SNSM on large data sets against large-scale clustering solvers, namely the \emph{K-means} and \emph{Mini-Batch K-Means} (\emph{MBKM}) algorithms~\cite{minibatchkmeans}.

\paragraph{Parameter tuning} In all the tests, we used the choice of parameters
\begin{equation*}\label{eq:paramtun}
	\bar{\tau}_0=1,\sigma=0.2, \beta=0.2 \text{ and } \gamma = 4,% \mM_0= 0.\tau_{\min}=10^{-4}
\end{equation*}
for the line search of all \salg, BDCA and RCSN (when required by the convergence result of the algorithm, we fixed
\(
\tau_{\min} = 10^{-4}  \text{ and } \tau_{\max} = 10^{8}
\)).
In addition, we always set \[{\mM_0=0} \text{ and }\mM=5\] in \salg, except when we wanted to run its monotone version resulting from taking $\mM=0$. Finally, the inertial parameter of iDCA was set to $0.99{\rho}/{2}$, where $\rho>0$ is the strong convexity parameter of the second function in the DC decomposition, to be specified in each experiment.

All the tests  were run on a computer with an AMD Ryzen 9 7950X 16-Core Processor, 4.50 GHz with 64 GB RAM, under Windows 10 (64-bit).

\subsection{Comparison with other optimization algorithms}\label{subsec:SOS}

The objective function $\varphi_{MSC}$ in~\eqref{eq:f_clust} for the minimum-sum-of-squares clustering problem also admits a decomposition as the difference of the convex functions
\[
g(X) := \frac{1}{p} \sum_{j=1}^p\sum_{t=1}^{\ell} \|x^t - a^j\|^2   \, \text{ and }
\, h(X):= \frac{1}{p} \sum_{j=1}^p \max_{l\in\{1,\ldots,\ell\}} \sum_{t=1,t\neq l}^{\ell}\|x^t-a^j\|^2.
\]
Further properties of this formulation can be consulted in, e.g.,~\cite{MR4139072,MR3306091}. This allows us to compare the implementation of \salg\ presented in Section~\ref{sect:appl} with that of DCA, iDCA, BDCA and RCSN over the previous decomposition, thereby extending the numerical study in~\cite[Section~5.1]{MR4078808}.

Note that strong convexity of $g$ is sufficient for the DCA framework. However, BDCA and iDCA additionally require the function $h$ to be strongly convex.
To ensure this property, we apply the standard DC regularization technique and  consider the equivalent decomposition
\[
\left(g + \tfrac{\rho}{2}\|\cdot\|^2\right) - \left(h + \tfrac{\rho}{2}\|\cdot\|^2\right),\quad \text{with } \rho = 0.1.
\]
On the other hand, observe that RCSN applied to~\eqref{eq:f_clust} results in a particular case of Algorithm~\ref{alg:SNSM} (with $m=0$) where the direction is taken as
$d_k = -[\nabla^2 g(X_k)]^{-1} w_k=- w_k/2$. %
%Hence, it can be viewed as a simple subgradient method with self-adaptive line search.
In contrast, we implemented \salg\ as described in Section~\ref{sect:appl}, with the search direction given by \eqref{eq:dkmsos} and $\alpha_k:=10^{-3}$ for all $k\in\N$.  We emphasize that RCSN does not coincide with the monotone implementation of \salg. Indeed, observe that  the first component $g$ is just a sum of squared Euclidean norms and hence the Hessian is simply a multiple of the identity matrix, which leads to a simple scaling of the gradient descent direction and fails to capture the true curvature of the problem. For this reason, we included in our comparison both the monotone version of \salg\ (with $\mM=0$) and its nonmonotone counterpart (with $\mM=5$).

We conducted our tests with  point configurations extracted from different datasets. Namely,  \textsc{Leaves}~\cite{mallah2013plant} and \textsc{Letters}~\cite{frey1991letter}, which are open-source real-world datasets available online at the University of California Irvine Machine Learning Repository~\cite{UCI}, as well as \textsc{Birch2} and \textsc{Birch3}, which are synthetic datasets originally introduced in~\cite{Birchsets}.  A brief description is displayed in Table~\ref{t:data-sets}.

\begin{table}[h!]
	\centering
	\begin{tblr}{
			colspec = {QS[table-format=2]S[table-format=5]S[table-format=2]},
			cell{1}{2} = {c},
			cell{1}{3} = {c},
			cell{5}{3} = {fg=MineShaft},
			hline{1-2} = {-}{0.08em},
			hline{9} = {-}{0.05em},
		}
		Data set        & {{{Dimension of space (s)}}} & {{{\# Data points (p)}}}  & {{{\# Clusters ($\ell$)}}}\\
		\textsc{Leaves} & 64 &1 600& 100 \\
		\textsc{Letters} & 16      & 20 000  & 26\\
		\textsc{Birch2} & 2 &100 000& 100 \\
		\textsc{Birch3} & 2 &100 000& 100 \\
		
		\bottomrule
	\end{tblr}
	\caption{Description of the datasets employed in our numerical tests}
	\label{t:data-sets}
\end{table}

For each dataset, we ran all algorithms from 10 different random starting points. In Table~\ref{t:mssc_1}, we report the average values for the number of iterations, number of function evaluations, running time, final function value $\varphi_{\mathrm{MSC}}(X_{\mathrm{out}})$, and infinity norm of the output subgradient $\|w_{\mathrm{out}}\|_{\infty}$ given by~\eqref{eq:subgradient} and~\eqref{eq:subgradient_j} (with ties broken by selecting the minimal index $t$).
For all the algorithms, we used the stopping criterion
\begin{equation}\label{eq:stopcrit}
	\max\left\{ \frac{\|X_k -X_{k-1}\|_F}{\max\{\|X_ {k-1}\|_F,1\}},  \frac{|\varphi_{MSC}(X_k)-\varphi_{MSC}(X_{k-1})|}{\max\{|\varphi_{MSC}(X_{k-1})|,1\}} \right\}\leq 10^{-4},
\end{equation}
where $X_k=(x_k^1,x_k^2,\ldots,x_k^\ell)$ denotes the $k$-th iteration of the algorithm and $\|\cdot\|_F$ is the Frobenius matrix norm.

The experimental results show that \salg\ beat  the other methods in every single aspect. Regarding the comparison between the monotone and nonmonotone versions of \salg, none of them seems to be consistently better than the other for this particular problem.

\begin{table}[!ht]
	\centering
	\resizebox{0.98\textwidth}{!}{
		\begin{tblr}{
				colspec = {
					l % Data sets
					l % Method
					r   % # Iterations
					r  % # f evaluations
					r % Time
					r                                      % phi_MSC
					r                                      % ||w_out||_infty
				},
				%row{1} = {cmd = \mbox},
				%colsep = 4pt
			}
			\toprule
			Data sets & Method & \mbox{\# Iter.} & \mbox{\# $f$ eval.} & \mbox{Time (s)} & \mbox{$\varphi_{\mathrm{MSC}}(X_{\mathrm{out}})$}  & $\|w_{\mathrm{out}}\|_{\infty}$  \\
			\midrule
			\textsc{Leaves} & DCA & 197 & 394 & 0.665 & 2.107e+06 & 3.810 \\
			& iDCA & 191 & 382 & 0.656 & 2.107e+06 & 3.648 \\
			& BDCA & 68 & 300 & 0.443 & 1.966e+06 & 1.700 \\
			& RCSN & 60 & 208 & 0.518 & 2.015e+06 & 2.665 \\
			& SNSM $\mM=0$ & 37 & 87 & 0.209 & \BB 7.131e+05 & \BB 0.373 \\
			& SNSM $\mM=5$ & \BB 25 & \BB 58 & \BB 0.138 & 8.123e+05 & 0.481 \\
			\hline
			\textsc{Letters} & DCA & 395 & 790 & 3.777 & 5.703e+01 & 0.014 \\
			& iDCA & 356 & 712 & 3.410 & 5.738e+01 & 0.015 \\
			& BDCA & 170 & 738 & 3.150 & 5.241e+01 & 0.004 \\
			& RCSN & 113 & 383 & 2.740 & 5.460e+01 & 0.005 \\
			& SNSM $\mM=0$ &\BB 54 & \BB 127 &  \BB 0.853& \BB 3.658e+01 & \BB 0.001 \\
			& SNSM $\mM=5$ & 56 & 132 & 0.876 & 3.744e+01 & \BB 0.001 \\
			\hline
			\textsc{Birch2} & DCA & 438 & 876 & 47.011 & 6.584e+07 & 61.203 \\
			& iDCA & 428 & 857 & 46.005 & 6.583e+07 & 60.105 \\
			& BDCA & 71 & 307 & 15.129 & 6.413e+07 & 20.817 \\
			& RCSN & 22 & 74 & 5.845 & 6.103e+07 & 12.174 \\
			& SNSM $\mM=0$ & \BB 17 & \BB 39 & \BB 2.911 & 5.110e+07 & 3.768 \\
			& SNSM $\mM=5$ & \BB 17 & 40 &  2.990 & \BB 5.109e+07 & \BB 2.938 \\
			\hline
			\textsc{Birch3} & DCA & 1047 & 2095 & 118.709 & 6.442e+08 & 174.630 \\
			& iDCA & 1030 & 2060 & 116.641 & 6.437e+08 & 172.626 \\
			& BDCA & 334 & 1440 & 73.927 & 5.714e+08 & 64.365 \\
			& RCSN & 95 & 325 & 27.553 & 5.409e+08 & 46.136 \\
			& SNSM $\mM=0$ & \BB 53 & \BB 127 & \BB 9.912 & 5.055e+08 & \BB 4.742 \\
			& SNSM $\mM=5$ & 58 & 137 & 10.697 & \BB 4.870e+08 & 5.793 \\
			
			\bottomrule
	\end{tblr}}
	\caption{Numerical results for the minimum sum-of-squares clustering model for the different datasets under the stopping criterion~\eqref{eq:stopcrit}, averaged over 10 random initial points}
	\label{t:mssc_1}
\end{table}

We observe in the last column of Table~\ref{t:mssc_1} that the stopping criterion~\eqref{eq:stopcrit} does not guarantee low values of $\|w_{\mathrm{out}}\|_\infty$ for the \textsc{Birch2} and \textsc{Birch3} datasets, which implies approximate stationarity. Although this criterion is common in practical applications, it is worth assessing how these methods perform when the stopping criterion is replaced by
\begin{equation}\label{eq:stopcrit_w}
	\|w_k\|_\infty\leq 10^{-3}.
\end{equation}
As neither DCA, iDCA, nor BDCA requires the computation of the subgradient $w_k$, we subtracted the time required to compute $w_k$ in these three methods to ensure a fair comparison in the numerical results reported in Table~\ref{t:mssc_2}. In spite of that, the running time of the three methods was significantly higher than that of SNSM, which was roughly the same for its monotone ($m=0$) and nonmonotone ($m=5$) versions.

\begin{table}[!ht]
	\centering
	\resizebox{0.98\textwidth}{!}{
		\begin{tblr}{
				colspec = {
					l % Data sets
					l % Method
					r   % # Iterations
					r  % # f evaluations
					r % Time
					r                                      % phi_MSC
				},
				%row{1} = {cmd = \mbox},
				%colsep = 4pt
			}
			\toprule
			Data sets & Method & \mbox{\# Iter.} & \mbox{\# $f$ eval.} & \mbox{Time (s)} & \mbox{$\varphi_{\mathrm{MSC}}(X_{\mathrm{out}})$}\\
			\midrule
			\textsc{Leaves} & DCA & 1386 & 2772 & 4.292 & 1.376e+06 \\
			& iDCA & 1311 & 2622 & 4.005 & 1.398e+06 \\
			& BDCA & 253 & 1107 & 1.459 & 1.314e+06 \\
			& RCSN & 197 & 676 & 0.925 & 1.318e+06 \\
			& SNSM $\mM=0$ & 42 & 98 & 0.213 & \BB 7.228e+05 \\
			& SNSM $\mM=5$ & \BB 41 & \BB 97 &  \BB 0.210 & 7.686e+05 \\
			\hline
			\textsc{Letters} & DCA & 2684 & 5368 & 22.573 & 4.932e+01 \\
			& iDCA & 2600 & 5200 & 21.874 & 4.932e+01 \\
			& BDCA & 316 & 1370 & 5.205 & 4.939e+01 \\
			& RCSN & 260 & 885 & 3.444 & 4.941e+01 \\
			& SNSM $\mM=0$ & 51 & 120 & 0.722 & \BB 3.674e+01 \\
			& SNSM $\mM=5$ & \BB 44 & \BB 103 & \BB 0.620 & 3.702e+01 \\
			\hline
			\textsc{Birch2} & DCA & 19240 & 38481 & 2013.283 & 6.823e+07 \\
			& iDCA & 18787 & 37575 & 1959.025 & 6.823e+07 \\
			& BDCA & 1832 & 7908 & 381.743 & 6.786e+07 \\
			& RCSN & 160 & 549 & 26.996 & 6.344e+07 \\
			& SNSM $\mM=0$ & \BB 27 & \BB 63 & \BB 4.625 & 5.617e+07 \\
			& SNSM $\mM=5$ & 29 & 69 & 5.012 & \BB 5.342e+07 \\
			\hline
			\textsc{Birch3} & DCA & 104648 & 209297 & 11169.871 & 5.180e+08 \\
			& iDCA & 102181 & 204363 & 10919.698 & 5.180e+08 \\
			& BDCA & 9921 & 42822 & 2079.365 & 5.201e+08 \\
			& RCSN & 832 & 2863 & 141.484 & 5.163e+08 \\
			& SNSM $\mM=0$ & \BB 82 & \BB 193 & \BB 14.526 & 5.126e+08 \\
			& SNSM $\mM=5$ & 84 & 198 & 14.832 & \BB 5.089e+08 \\
			\bottomrule
	\end{tblr}}
	\caption{Numerical results for the minimum sum-of-squares clustering model for the different datasets under the stopping criterion~\eqref{eq:stopcrit_w}, averaged over 10 random initial points}
	\label{t:mssc_2}
\end{table}

We additionally explored how the precision in~\eqref{eq:stopcrit_w} affects the running time and observed two different behaviors across the four datasets. The \textsc{Leaves}, \textsc{Birch2} and \textsc{Birch3} datasets were not strongly affected by the choice of $10^{-3}$. A higher or lower value did not have a major impact on the running time, as shown in Figure~\ref{fig:norm_a} for a particular instance of \textsc{Birch3}. On the other hand, for the \textsc{Letters} dataset, relaxing the tolerance in~\eqref{eq:stopcrit_w} to a value of $10^{-1}$ or even $10^{-2}$  would have drastically reduced the running time of the DCA-based algorithms (see Figure~\ref{fig:norm_b}). We currently lack a theoretical explanation for this disparity. In any case, this experiment indicates that the stopping criterion~\eqref{eq:stopcrit_w}, which is a natural choice for Algorithms~\ref{alg:1} and~\ref{alg:SNSM}, provides an additional advantage by ensuring approximate criticality of the output without incurring any increase in computational time.

\begin{figure}[h!]\centering
	\begin{subfigure}{\textwidth}
		\centering
		\includegraphics[width=\textwidth]{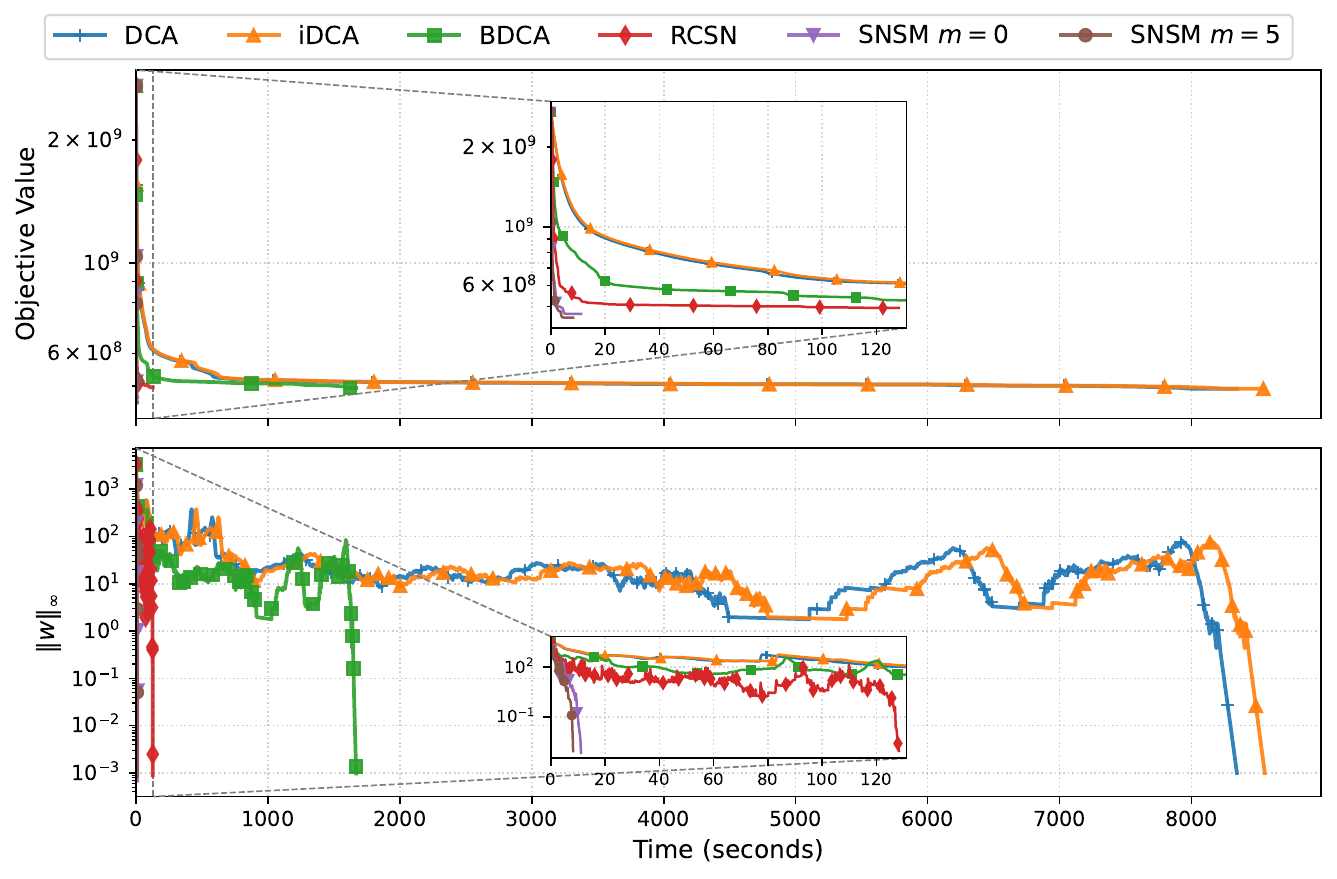}
		\caption{\textsc{Birch3} dataset}\label{fig:norm_a}
	\end{subfigure}\vspace{4pt}
	\begin{subfigure}{\textwidth}
		\centering
		\includegraphics[width=\textwidth]{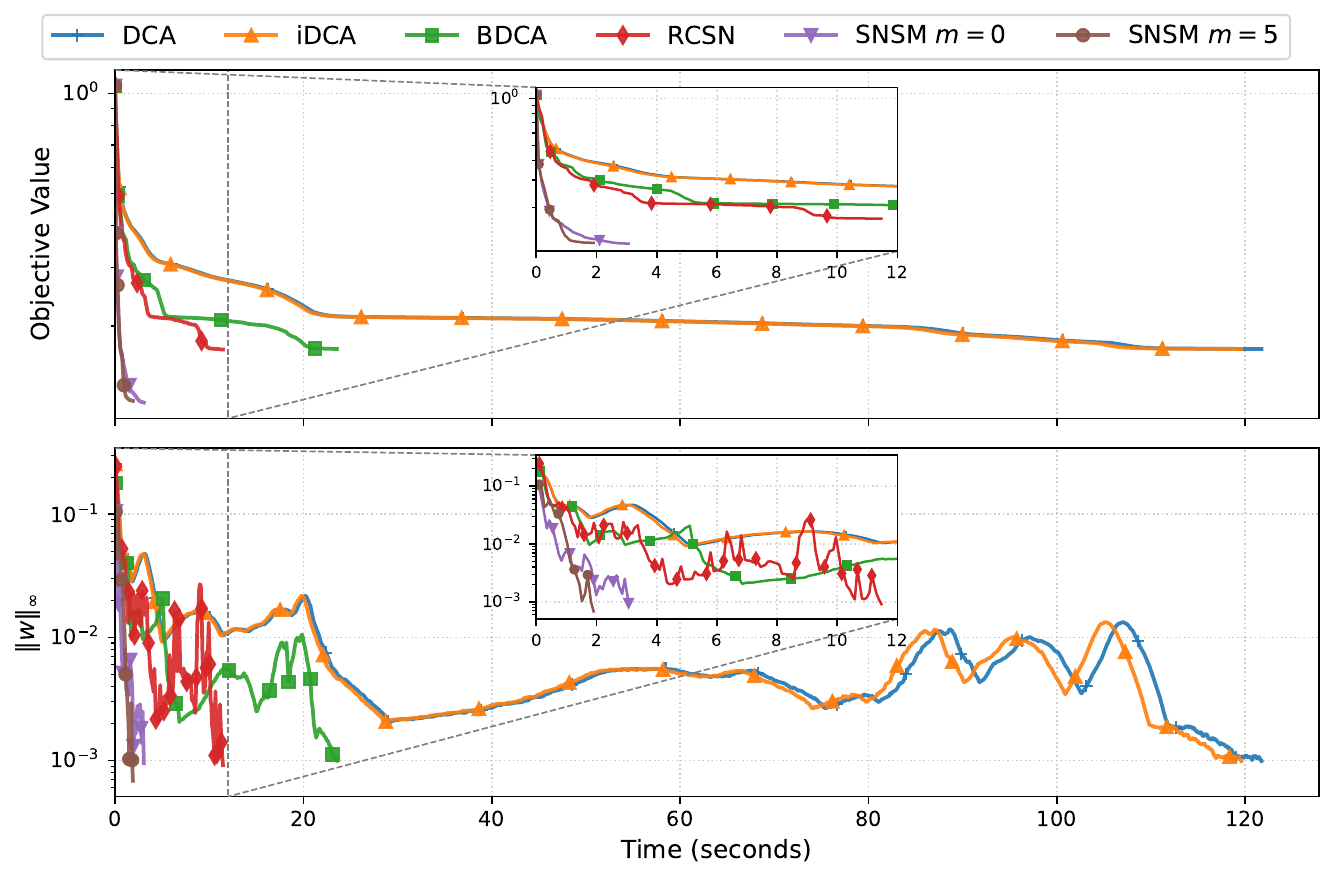}
		\caption{\textsc{Letters} dataset}\label{fig:norm_b}
	\end{subfigure}
	\caption{Evolution of the objective value and subgradient norm (on a logarithmic scale) for a particular instance of the \textsc{Birch3} (top) and \textsc{Letters} (bottom) datasets. Inset plots highlight the behavior during the initial seconds}\label{fig:norm}
\end{figure}

\subsection{Comparison  with  large-scale clustering solvers}\label{subsec:kmeans}

While SNSM was not primarily developed to compete with algorithms tailored for large datasets, we performed further numerical experiments to assess its scalability. Specifically, we compared our proposal against the standard \mbox{K-means} and Mini-Batch K-Means (MBKM) algorithms, as implemented in the \texttt{scikit-learn} Python library~\cite{scikit}.
Because \texttt{scikit-learn}'s K-means and MBKM rely on highly optimized Cython and C-extensions, we implemented our algorithm using Numba's Just-In-Time (JIT) compiler to guarantee a more fair runtime comparison. We also replaced our stopping criterion~\eqref{eq:stopcrit} by
\begin{equation}\label{eq:stopcrit2}
	\frac{\|X_k -X_{k-1}\|_F}{\max\{\|X_{k-1}\|_F,1\}}\leq 10^{-4},
\end{equation}
to match the one used for \emph{K-means} in this library (note that MBKM relies on a different internal stopping condition, despite using the same tolerance parameter; see the \texttt{scikit-learn} documentation).

First, we compared the algorithms' performance for two real-world datasets included in the \texttt{scikit-learn} library, namely, \emph{covtype} ($s=54$, \mbox{$p=581\,012$}, $\ell=7$) and \emph{kddcup99} ($s=41$, \mbox{$p=494\,021$}, \mbox{$\ell=5$}). For each dataset, we  generated 10 sets of initial cluster centers using \texttt{scikit-learn}'s \texttt{kmeans\_plusplus} function (the standard warm-start technique used for K-means and MBKM) with different random seeds. Subsequently, we executed K-means, MBKM and SNSM (for $m=0$ and $m=5$) using each of these initializations. The batch size for MBKM  was set to 5000.\\

We used the following three metrics to evaluate clustering quality:
\begin{itemize}
	\item The \emph{Rand Index} (\emph{RI}), which measures the similarity between the clustering assignments and the ground truth sample labels. A perfect labeling achieves an RI of $1$, whereas poorly agreeing assignments yield lower scores.
	\item The \emph{Silhouette Coefficient} (\emph{SC}), which for a single data point is given by the ratio $(sc_1 - sc_2)/\max\{sc_1,sc_2\}$, where $sc_1$ is the mean distance between the data point and all other points in the next nearest cluster, and $sc_2$ is the mean distance between the data point and all other points in the same cluster. The overall SC for a dataset is computed as the average across all samples. Higher values indicate better-defined clusters. Due to its computational complexity, the SC is estimated using a random subsample of $50\,000$ points.
	\item The \emph{Calinski--Harabasz index} (\emph{CH}), which is defined as the ratio of the sum of between-clusters dispersion and of within-cluster dispersion. Higher scores indicate dense and well-separated clusters. For further details, we refer the reader to the \href{https://scikit-learn.org/stable/modules/clustering.html\#overview-of-clustering-methods}{\texttt{scikit-learn} documentation}.
\end{itemize}
Unlike the RI, neither the SC nor the CH measures require knowledge of the ground truth labels. For all three metrics, higher values indicate superior clustering quality.

Table~\ref{t:solvers} summarizes the performance of the evaluated solvers. We observe that SNSM exhibits a competitive computational cost compared to the \texttt{scikit-learn} implementations,  converging faster in some instances. In particular, for the \textsc{kddcup99} dataset, SNSM attained the lowest running time, while maintaining a clustering quality comparable to that of K-means and significantly better than MBKM in terms of function value and clustering metrics. The nonmonotone version of SNSM ($m=5$) was faster than its monotone counterpart. Regarding clustering quality, the results of SNSM (both $m=0$ and $m=5$)  were similar to or better than those obtained by K-means and MBKM across the considered metrics.  Recall that MBKM's random mini-batch sampling mechanism struggles on  datasets with  high concentration of exact duplicates, such as \textsc{kddcup99}, which may explain the poor values of the objective function and the SC and CH metrics.

\begin{table}[!ht]
	\centering
	\resizebox{.95\textwidth}{!}{%
		\begin{tblr}{
				width=\linewidth,
				colspec={Q[l] Q[l] Q[c] Q[c] Q[c] Q[c] Q[c]},
				cell{1}{1,2,3,4}={r=2}{c},
				cell{1}{5}={c=3}{c},
				cell{2}{5,6,7}={c},
				hline{1}={-}{0.08em},
				hline{2}={5-7}{0.02em},
				hline{3}={-}{0.05em},
			}
			Data sets & Method & Time & Function value & Clustering quality & & \\
			& & & & RI & SC & CH \\
			
			\textsc{covtype} & K-means & 2.4338 & {\BB 7.7155e+05} & {0.5762} & {\BB 0.3302} & {\BB 4.4622e+05} \\
			& MBKM & {\BB 0.2810} & {8.1872e+05} & \BB 0.5821 & 0.2940 & {4.1713e+05} \\
			& SNSM $m=0$ & 0.4990 & {\BB 7.7155e+05} & {0.5762} & {\BB 0.3302} & {\BB 4.4622e+05} \\
			& SNSM $m=5$ & 0.4681 & {\BB 7.7155e+05} & {0.5762} & {\BB 0.3302} & {\BB 4.4622e+05} \\
			
			\hline
			\textsc{kddcup99} & K-means & 0.2957 & {\BB 2.1300e+08} & 0.6671 & {\BB 0.9988} & {\BB 5.9020e+08} \\
			& MBKM & 0.1991 & {9.7766e+11} & 0.6076 & 0.6164 & 4.6068e+01\\
			& SNSM $m=0$ & 0.0204 & {2.3124e+08} & {\BB 0.6678} & 0.9950 & {5.6321e+08} \\
			& SNSM $m=5$ & {\BB 0.0166} & {2.3124e+08} & {\BB 0.6678} & 0.9950 & {5.6321e+08} \\
			
			\hline
		\end{tblr}%
	}
	\caption{Performance comparison between SNSM, K-means, and MBKM for clustering on two real-world datasets}
	\label{t:solvers}
\end{table}

Lastly, we conducted a stress test to evaluate the scalability of SNSM with respect to problem size. We used the \texttt{make\_blobs} function from \texttt{scikit-learn} to generate synthetic datasets under three distinct configurations. In each configuration, we varied only a single problem parameter while keeping the others fixed:
\begin{itemize}
	\item \textbf{Data points}: For $s=10$ and $\ell = 20$, we set $p\in\bigl\{10^5\cdot j : 1\leq j \leq 20  \bigr\}$.
	
	\item \textbf{Data dimension}: For $p=10^6$ and $\ell = 20$, we set $s\in \bigl\{10^2\cdot j : 1\leq j\leq 20  \bigr\}$.
	
	\item \textbf{Clusters}: For $s=100$ and $p=10^6$, we set $\ell \in \bigl\{ 20\cdot j : 1\leq j \leq 20  \bigr\}$.
\end{itemize}

For each configuration, we ran K-means, MBKM, and SNSM ($m=5$) across 10 random seeds. Figure~\ref{fig:stress_tests} presents the average running times on a logarithmic scale. These results highlight SNSM's strong performance compared to the other two algorithms and indicate that SNSM scales at the same rate as the K-means algorithm for all the tested parameters.

\begin{figure}[h!]\centering
	\includegraphics[width=0.5\textwidth]{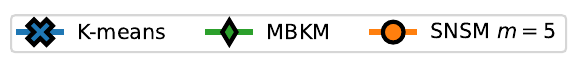}\\
	\includegraphics[width=.99\textwidth]{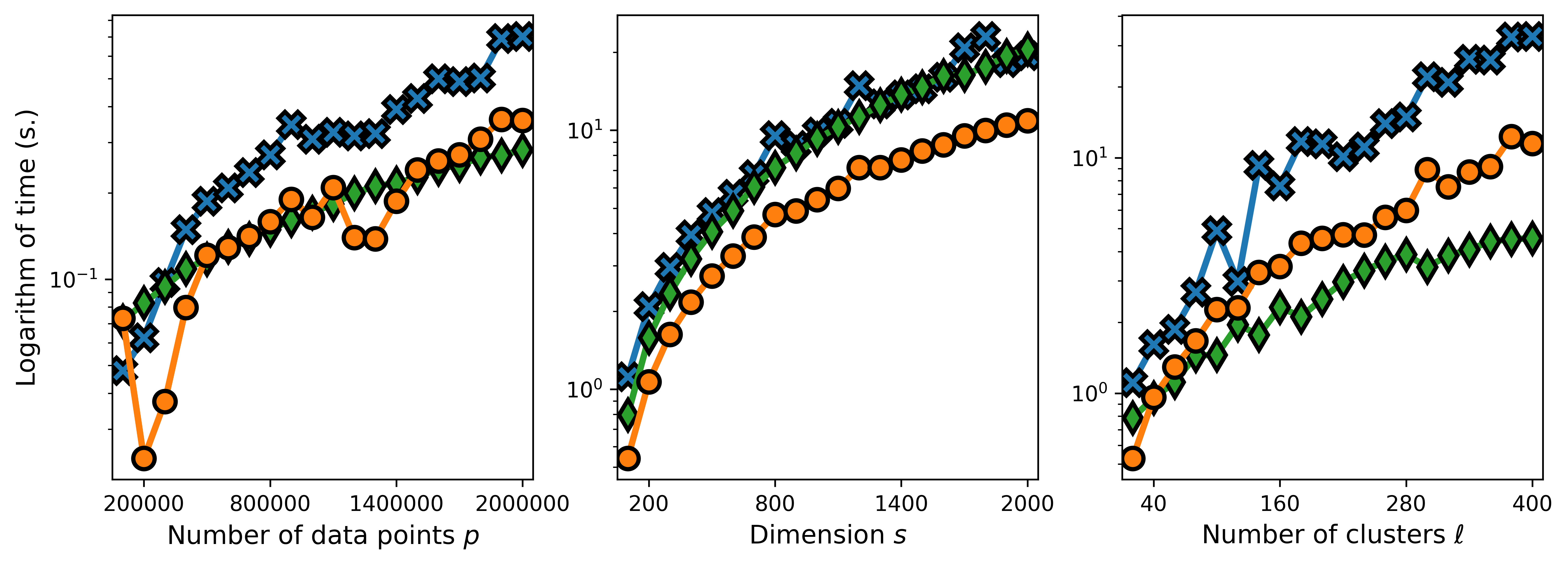}
	\caption{Scalability of SNSM, K-means and MBKM under increasing problem sizes}\label{fig:stress_tests}
\end{figure}

\section{Conclusion}\label{sect:conc}

In this work, we proposed a general numerical method for addressing the minimization of upper-$\mathcal{C}^2$ functions.  We showed that this family of functions is characterized by the verification of a nonsmooth local version of the descent lemma. They naturally appear in multiple situations of interest in optimization, such as difference programming and augmented Lagrangian methods.

We introduced in Algorithm~\ref{alg:1} a general nonmonotone subgradient method for the minimization of upper-$\mathcal{C}^2$ functions. This algorithm relies on the computation of a (Clarke) subgradient and the employment of a nonmonotone Armijo-type line search, using a direction that can be conveniently chosen, which allows the use of second-order information. In the main theoretical result, Theorem~\ref{t:conv}, we proved subsequential convergence of the sequence generated by the algorithm to a stationary point of the problem. In addition, we introduced in Algorithm~\ref{alg:SNSM} a specification of Algorithm~\ref{alg:1}, named the \emph{Self-adaptive Nonmonotone Subgradient Method} (\salg),  that incorporates a procedure for automatically selecting the parameters of the nonmonotone line search.

To test the performance of \salg, we conducted numerical comparisons against various specialized algorithms for difference programming, specifically, we considered the minimum sum-of-squares clustering problem.

Overall, the flexibility in the choice of search directions, together with the nonmonotone line search strategy, seemed to provide a clear advantage in terms of convergence speed and quality of solutions. Remarkably, SNSM was competitive against the clustering solvers K-means and Mini-Batch K-Means, and it scaled well with increasing problem sizes.

\section*{Declarations}

\paragraph{Funding} This research was partially supported by PID2022-136399NB-C21, PID2025-168246NB-I00 and PCI2024-155024-2 funded  by MICIU/AEI/10.13039/501100011033 and by ERDF/EU, as well as by CIGE/2026/80 by Conselleria d'Educaci\'o, Cultura i Universitats, Generalitat Valenciana. Research of the third author was partially supported by  Centro de
Modelamiento Matem\'{a}tico (CMM), ACE210010 and FB210005, BASAL funds for
center of excellence and ANID-Chile: MATH-AMSUD 23-MATH-09 and
MATH-AMSUD 23-MATH-17, ECOS-ANID ECOS320027, Fondecyt Regular 1220886,
Fondecyt Regular 1240335, Fondecyt Regular 1240120 and Exploraci\'on  13220097.
The fourth author was partially supported by Fondecyt Postdoctorado 3250039 from ANID Chile.

\paragraph{Competing interests} The authors declare no competing interests.

\bibliographystyle{siam}

\end{document}